\documentclass[12pt, a4paper]{amsart}
\usepackage{amsmath}
\usepackage{amsfonts}
\usepackage{amssymb}
\usepackage[all]{xy}           %xypic macro for latex2.09
\usepackage{bbding}
\usepackage{txfonts}
\usepackage{amscd}
\usepackage{amsthm}
\usepackage[T1]{fontenc}
\usepackage[shortlabels]{enumitem}
\usepackage{ifpdf}
\ifpdf
  \usepackage[colorlinks,final,backref=page,hyperindex]{hyperref}
\else
  \usepackage[colorlinks,final,backref=page,hyperindex]{hyperref}
\fi
\usepackage{tikz}
\usepackage[active]{srcltx}

\makeatletter
\numberwithin{equation}{section}

\newtheorem{thm}{Theorem}[section]
\newtheorem{lem}[thm]{Lemma}
\newtheorem{cor}[thm]{Corollary}
\newtheorem{pro}[thm]{Proposition}
\newtheorem{ex}[thm]{Example}
\newtheorem{rmk}[thm]{Remark}
\newtheorem{defi}[thm]{Definition}

\newcommand{\fl}{\mathbf l}
\newcommand{\fr}{\mathbf r}

\newcommand{\g}{\mathfrak{g}}

\newcommand{\kt}{\mathfrak{t}}
\newcommand{\kl}{\mathfrak{l}}
\newcommand{\kr}{\mathfrak{r}}

\newcommand{\bz}{\mathbb{Z}}

\newcommand{\ad}{\mathrm{ad}}

\newcommand{\CYBE}{\mathrm{CYBE}}
\newcommand{\LYBE}{\mathrm{LYBE}}
\newcommand{\ZYBE}{\mathrm{ZYBE}}
\newcommand{\End}{\mathrm{End}}
\newcommand{\id}{\mathrm{id}}

\newcommand{\tr}{\triangleright}
\newcommand{\tl}{\triangleleft}

\begin{document}

\title[Lie bialgebras via Zinbiel and Leibniz bialgebras]
{Lie bialgebras constructed from Zinbiel bialgebras and Leibniz bialgebras}

\author[B.~Hou]{Bo Hou}
\address{School of Mathematics and Statistics, Henan University, Kaifeng 475004, China}
\email{bohou1981@163.com}

\author[Y.~Lin]{Yuanchang Lin}
\address{School of Mathematics, North University of China, Taiyuan 030051, China}
\email{linyuanchang@mail.nankai.edu.cn}

\date{\today}

\begin{abstract}
There is a Lie algebra structure on the tensor product of a Leibniz algebra and a Zinbiel algebra for the operads of Leibniz algebras and Zinbiel algebras are Koszul dual. 
In this paper, we extend such conclusion to the context of bialgebras.
We show that there is a Lie bialgebra structure on the tensor product of a Leibniz bialgebra and a quadratic Zinbiel algebra, while there is an infinite-dimensional Lie bialgebra structure on the tensor product of a Zinbiel bialgebra and a quadratic $\bz$-graded Leibniz algebra.
For a special quadratic $\bz$-graded Leibniz algebra, the property that its tensor product with a Zinbiel bialgebra forms a Lie bialgebra characterizes the Zinbiel bialgebra.
By examining the relationship between solutions of the classical Yang-Baxter equation in a Zinbiel algebra (resp. a Leibniz algebra) and solutions
of the classical Yang-Baxter equation in the induced Lie algebra, we prove that the induced
Lie bialgebra is quasi-triangular (resp. triangular, factorizable) whenever the original
Zinbiel bialgebra (resp. Leibniz bialgebra) is quasi-triangular (resp. triangular,
factorizable). 
Finally, we present a construction of a quasi-Frobenius Lie algebra on the tensor product of a quasi-Frobenius Zinbiel algebra and a quadratic Leibniz algebra.
\end{abstract}

\keywords{Lie bialgebra, Zinbiel bialgebra, Leibniz bialgebra, classical Yang-Baxter equation,
$\mathcal{O}$-operator, quasi-Frobenius Lie algebra}
\subjclass[2020]{
17A32, %Leibniz algebras
%17A30, %Nonassociative algebras satisfying other identities
%17D25, %Lie-admissible algebras
%17A36, %Automorphisms, derivations, other operators (nonassociative rings and algebras)
17A60, %Structure theory for nonassociative algebras
%17B10, %Representations, algebraic theory
17B38, %Yang-Baxter equations and Rota-Baxter operators
%17B40, %Automorphisms,derivations,other operators
%17B60, %Lie (super)algebras associated to other structures (associative, Jordan, ect.)
17B62. %Lie bialgebras; Lie coalgebras
}
\maketitle
\vspace{-10mm}
\tableofcontents %Ŀ¼

\allowdisplaybreaks

\vspace{-10mm}
%%%%%%%%%%%%%%%%%%%%%%%%%%%%%%%%%%%%%%%%%%%%%%%%%%%%%%%%%%%%%%%%%%%%%%%%%%%%%%%%%
%    section  1   Introduction
%%%%%%%%%%%%%%%%%%%%%%%%%%%%%%%%%%%%%%%%%%%%%%%%%%%%%%%%%%%%%%%%%%%%%%%%%%%%%%%%%%%%%%
\section{Introduction}\label{sec:intr}
There is a Lie algebra structure on the tensor product of a Leibniz algebra and a Zinbiel algebra, 
which fits into the Koszul duality interpretation. 
In this paper, we extend such construction in the context of bialgebras.
Specifically, we construct Lie bialgebras on the tensor product of Leibniz bialgebras and quadratic Zinbiel algebras, and infinite-dimensional Lie bialgebras from the affinization of Zinbiel bialgebras by quadratic $\bz$-graded Leibniz algebras respectively.
We also show that the constructed Lie bialgebra inherits the quasi-triangular (resp. triangular, factorizable) property from the original Zinbiel (resp. Leibniz) bialgebra.

%%%%%%%%%%%%%%%%%%%%%%%%%%%%%%%%%%%%%%%%%%%%%%%%%%%%%%%%%%%%%%%%%%%%%%%%%%%%%%%%
\subsection{Lie bialgebras and their construction}
A bialgebra structure on a given algebra is obtained as a coalgebra structure that induces the same algebra structure on the dual space, subject to a set of compatibility conditions between the product and coproduct.
One of the most famous examples of bialgebra is the Lie bialgebra \cite{Dri}.
Lie bialgebras can be viewed as an infinitesimalization of Poisson-Lie groups. 
This concept was first introduced by Drinfeld in the context of the Yang-Baxter equations and quantum groups. 
Since then, it has found applications in many other areas of mathematics and mathematical physics, such as the theory of Hopf algebra deformations of universal enveloping algebras \cite{ES}, string topology and symplectic field theory \cite{CFL}, the Goldman-Turaev theory of free loops in Riemann surfaces with boundaries \cite{Gol}, and beyond.
There have been many bialgebra theories for other algebra structures that essentially follow the approach of Lie bialgebras, including: antisymmetric infinitesimal bialgebras \cite{Agu,Bai1}, left-symmetric bialgebras \cite{Bai}, Jordan bialgebras \cite{Zhe}, Novikov bialgebras \cite{HBG}, perm bialgebras \cite{Hou,LZB}, Leibniz bialgebras \cite{TS,BLST}, anti-pre-Lie Poisson bialgebras \cite{LB}, Zinbiel bialgebras \cite{Wan}, and so on.
Recently, the extension problem of Lie bialgebras \cite{Hon} and factorizability of Lie bialgebras \cite{LS} have been further studied.

A well-known phenomenon from operadic theory states that if $A$ and $B$ are algebras over binary quadratic operads that are Koszul dual to each other, then their tensor product $A\otimes B$ can be equipped with a bracket $[-,-]$ such that $(A\otimes B, [-,-])$ forms a Lie algebra \cite{GK}. 
Based on this fact, Hong, Bai, and Guo introduced the notion of a Novikov bialgebra to construct infinite-dimensional Lie bialgebras via Novikov bialgebra affinization \cite{HBG}. 
Lin, Zhou, and Bai later showed that the tensor product of a perm bialgebra and a quadratic pre-Lie algebra (resp. a pre-Lie bialgebra and a quadratic perm algebra) admits a Lie bialgebra structure \cite{LZB}. 
Notably, the operads of perm algebras and pre-Lie algebras are Koszul dual to each other \cite{Lo1}. 
In this paper, we construct quasi-triangular (resp. triangular, factorizable) Lie bialgebras using Leibniz bialgebras, and we also construct infinite-dimensional Lie bialgebras via the affinization of Zinbiel bialgebras.

%%%%%%%%%%%%%%%%%%%%%%%%%%%%%%%%%%%%%%%%%%%%%%%%%%%%%%%%%%%%%%%%%%%%%%%%%%%%%%%%
\subsection{Constructions of Lie bialgebras from Leibniz bialgebras}
Leibniz algebras were first discovered by Bloh under the name of D-algebras \cite{Blo} and later introduced by Loday \cite{LP} with the motivation of studying periodicity in algebraic K-theory.
The (co)homology and homotopy theories of Leibniz algebras were established in \cite{Gao,LP}. 
Recently, Leibniz algebras were studied from different aspects due to applications in both mathematics and physics, such as integration, deformation quantization, rational homotopy theory,
and higher gauge theories \cite{BW,Cov,DW,Liv}. 
In \cite{TS}, the bialgebra theory and the classical Yang-Baxter equation for Leibniz algebras were introduced, where Tang and Sheng established the equivalence between Leibniz bialgebras and Manin triples of Leibniz algebras, and subsequently applied twisting theories and relative Rota-Baxter operators to study triangular Leibniz bialgebras.
A one-to-one correspondence between factorizable Leibniz bialgebras and quadratic Rota-Baxter Leibniz algebras of nonzero weights was established by Bai, Liu, Sheng, and Tang \cite{BLST}, who also introduced the notions of quasi-triangular and factorizable Leibniz bialgebras.
Here, we use quasi-triangular (resp. triangular, factorizable) Leibniz bialgebras to construct quasi-triangular (resp. triangular, factorizable) Lie bialgebras. 

By showing that the tensor product of a Leibniz (co)algebra and a Zinbiel (co)algebra admits a Lie (co)algebra structure, we are able to construct Lie bialgebras on the tensor product of Leibniz bialgebras and quadratic Zinbiel algebras. 
Manin triples of Lie algebras (resp. Leibniz algebras) provide an equivalent description of Lie bialgebras (resp. Leibniz bialgebras). 
We present a natural construction of Manin triples of Lie algebras from those of Leibniz algebras. 
By establishing the relationship between solutions of the classical Yang-Baxter equation ($\CYBE$) in the induced Lie algebra and solutions of the Leibniz Yang-Baxter equation ($\LYBE$) in the original Leibniz algebra, we can construct quasi-triangular (resp. triangular, factorizable) Lie bialgebras from the corresponding Leibniz bialgebras.
Moreover, we interpret symmetric solutions of the $\LYBE$ in a Leibniz algebra and skew-symmetric solutions of the $\CYBE$ in the induced Lie algebra via $\mathcal{O}$-operators on Leibniz algebras and Lie algebras, respectively. 
We summarize these results in the following diagram:
$$
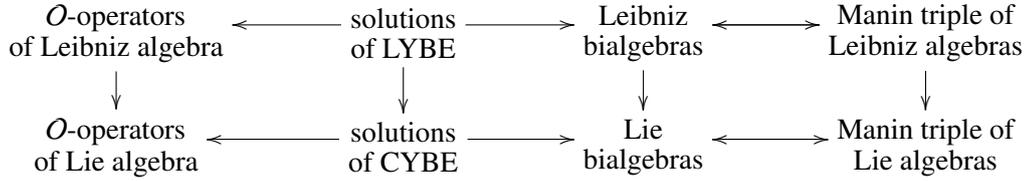
\xymatrix@C=1.4cm@R=0.5cm{
\txt{\small $\mathcal{O}$-operators \\ \small of Leibniz algebra} \ar[d] &
\txt{\small solutions \\ \small of $\LYBE$ } \ar[d]\ar[r]\ar[l]&
\txt{\small Leibniz \\ \small bialgebras} \ar[d]\ar[r]&
\txt{\small Manin triple of \\ \small Leibniz algebras}\ar[d]\ar[l] \\
\txt{\small $\mathcal{O}$-operators \\ \small of Lie algebra} &
\txt{\small solutions \\ \small of $\CYBE$ } \ar[r]\ar[l]&
\txt{\small Lie \\ \small bialgebras} \ar[r]&
\txt{\small Manin triple of \\ \small Lie algebras}\ar[l]}
$$

%%%%%%%%%%%%%%%%%%%%%%%%%%%%%%%%%%%%%%%%%%%%%%%%%%%%%%%%%%%%%%%%%%%%%%%%%%%%%%%%
\subsection{Zinbiel bialgebra and their affinization constructions of Lie bialgebras}
Zinbiel algebras were introduced by Loday in \cite{Lo1}. 
Under Koszul duality, the operad of Zinbiel algebras is dual to that of Leibniz algebras. Zinbiel algebras are also known as pre-commutative algebras or chronological algebras \cite{Kaw}, and they are equivalent to commutative dendriform algebras \cite{Agu}. 
They play an important role in the definition of pre-Gerstenhaber algebras \cite{AAC}. Recently, Zinbiel algebras have appeared in the study of rack cohomology \cite{CFLM}, number theory \cite{Cha}, and in the construction of Cartesian differential categories \cite{IP}. 
Moreover, Wang developed a bialgebra theory for Zinbiel algebras and studied quasi-triangular and factorizable Zinbiel bialgebras \cite{Wan}.

A crucial construction of infinite-dimensional Lie algebras is known as the process of affinization. 
Roughly speaking, the affinization of a given algebra structure consists of defining an infinite-dimensional algebra structure and then obtaining another algebra structure on the tensor product of the original algebra with the infinite-dimensional one, which in turn can resolve the original algebra structure. 
The affinization of Novikov bialgebras, perm bialgebras, and dendriform $D$-bialgebras has been studied in \cite{HBG}, \cite{LZB}, and \cite{Hou1}, respectively, using an infinite-dimensional algebra structure on the vector space of Laurent polynomials. 
In this paper, we define a $\bz$-graded Leibniz algebra $\widehat{V}_{4}$ in Example \ref{ex:gr-Leib}, and obtain an affinization characterization of Zinbiel algebras by $(\widehat{V}_{4}, \circ)$.
Moreover, we show that there exists a completed
Lie coalgebra structure on the tensor product of a Zinbiel coalgebra and a completed
Leibniz coalgebra, which could give a characterization of the Zinbiel coalgebra by
its affinization with a completed Leibniz coalgebra structure on the vector space of
$\widehat{V}_{4}$. 
Therefore, we show that there is a natural completed Lie bialgebra structure on the tensor product of a Zinbiel bialgebra and a quadratic $\bz$-graded Leibniz algebra. 
By the quadratic $\bz$-graded Leibniz algebra $(\widehat{V}_{4}, \circ, \omega)$, an affinization characterization of a Leibniz bialgebra is given.

\smallskip
\noindent
{\bf Theorem } [Theorems~\ref{thm:bialg} and \ref{thm:indu-sLiebia-zin}]
{\it Let $(A, \cdot, \Delta)$ be a Zinbiel bialgebra, $(B=\oplus_{i\in\bz}B_{i},
\circ, \omega)$ be a quadratic $\bz$-graded Leibniz algebra and
$(\g:=A\otimes B,\; [-,-])$ be the induced Lie algebra from $(A, \cdot)$ and $(B, \circ)$. Define two linear maps $\vartheta:
B\rightarrow B\,\hat{\otimes}\,B$ by Eq. \eqref{quad-dual} and $\delta: \g\rightarrow
\g\,\hat{\otimes}\,\g$ by Eq. \eqref{colie}. Then $(\g, [-, -], \delta)$ is a completed
Lie bialgebra. In particular, if $(B=\oplus_{i\in\bz}B_{i}, \circ, \omega)$ is the quadratic
$\bz$-graded Leibniz algebra $(\widehat{V}_{4}, \circ, \omega)$ given in Example
\ref{ex:Leib-quad}, then $(\g, [-, -], \delta)$ is a completed Lie bialgebra if and
only if $(A, \cdot, \Delta)$ is a Zinbiel bialgebra. Moreover, we have
\begin{enumerate}\itemsep=0pt
\item[$(i)$] $(A\otimes B, [-,-], \delta)$ is quasi-triangular if
     $(A, \cdot, \Delta)$ is quasi-triangular;
\item[$(ii)$] $(A\otimes B, [-,-], \delta)$ is triangular if
     $(A, \cdot, \Delta)$ is triangular.
\end{enumerate}}

\smallskip\noindent
In addition, we use quasi-Frobenius Zinbiel algebras to construct quasi-Frobenius
$\bz$-graded Lie algebras, thereby providing many examples of quasi-Frobenius Lie algebras.

%%%%%%%%%%%%%%%%%%%%%%%%%%%%%%%%%%%%%%%%%%%%%%%%%%%%%%%%%%%%%%%%%%%%%%%%%%%%%%%%
\subsection{Outline of the paper}
The paper is organized as follows.
In Section~\ref{sec:zinb}, we recall the notions of Zinbiel algebras and Leibniz algebras,
and show that there are Lie (co)algebra structures on the tensor product of Zinbiel (co)algebras and Leibniz (co)algebras.

In Section~\ref{sec:Leib-Liebia}, 
by giving a Zinbiel coalgebra structure on the vector space of a quadratic Zinbiel algebra, we construct a Lie bialgebra on the tensor product of a Leibniz bialgebra and a quadratic Zinbiel algebra in Theorem~\ref{thm:liebia-LZ}. 
We also present a natural construction of Manin triples of Lie algebras from Manin triples of Leibniz algebras in Proposition~\ref{pro:comm}. 
We construct solutions of the $\CYBE$ in the induced Lie algebra from solutions of the $\LYBE$ in the original Leibniz algebra, and prove that 
the former solution is skew-symmetric if the latter solution is symmetric.
Moreover, the symmetric part of the former solution is Lie-invariant provided that the skew-symmetric part of the latter is Leib-invariant. 
Furthermore, we show in Theorem~\ref{thm:indu-sLiebia-lei} that the induced Lie bialgebra is quasi-triangular (resp. 
triangular, factorizable) if the original Leibniz bialgebra is quasi-triangular 
(resp. triangular, factorizable).

In Section~\ref{sec:inf-lie}, we show that there is a natural completed Lie bialgebra structure on the tensor product of a Zinbiel bialgebra and a quadratic $\bz$-graded Leibniz algebra. 
By a special quadratic $\bz$-graded Leibniz algebra, an affinization characterization of Zinbiel bialgebras is given in Theorem~\ref{thm:bialg}.
Furthermore, from solutions of the Zinbiel Yang-Baxter equation ($\ZYBE$) in a Zinbiel algebra, we construct completed solutions of the $\CYBE$ in the induced Lie algebra. 
The latter solutions can be regarded as an  affinization of the former solutions (see Proposition~\ref{pro:ZYBE-CYBE} and Corollary~\ref{cor:sZYBE-sCYBE}). 
Therefore, we obtain that the induced completed Lie algebra is quasi-triangular (resp. triangular) whenever the original Zinbiel bialgebra is quasi-triangular (resp. triangular). 
Finally, constructions of quasi-Frobenius $\bz$-graded Lie algebras from quasi-Frobenius Zinbiel algebras are obtained in Proposition~\ref{pro:quasi-frob}.

Throughout this paper, we work over a base field $\Bbbk$ of characteristic $0$, and all vector spaces and algebras are assumed to be finite-dimensional unless otherwise specified. 
We adopt the following conventions and notations.
\begin{enumerate}
	\item 
	Let $(A, \diamond)$ be a vector space equipped with a binary operation $\diamond: A \otimes A \to A$.
	Let $\fl_\diamond(a)$ and $\fr_\diamond(a)$ denote the left and right multiplication operators, that is 
	\begin{equation*}
		\fl_\diamond(a)b = \fr_\diamond(b)a = a \diamond b, \;\; \forall a, b \in A.
	\end{equation*}
	We also simply denote them by $\fl(a)$ and $\fr(a)$ respectively without confusion.
	If $(A, [\ ,\ ])$ is a Lie algebra, we let $\ad_{[,]}(a) = \ad(a)$ denote the adjoint operator, that is 
	\begin{equation*}
		\ad_{[,]}(a)b = \ad(a) b = [a, b], \;\; \forall a, b \in A.
	\end{equation*}
	
	\item 
	Let $V$ be a vector space. 
	Denote the flip operator by $\tau: V\otimes V\rightarrow V\otimes V$, which is defined by 
	\begin{equation*}
		\tau(u\otimes v) = v \otimes u, \;\; \forall u, v \in V.
	\end{equation*}
	
	\item 
	Let $(A, \diamond)$ be a vector space equipped with a binary operation $\diamond: A \otimes A \to A$.
	Let $r = \sum_{i} a_i \otimes b_i \in A \otimes A$.
	Set 
	\begin{equation*}
		r_{12} = \sum\nolimits_{i} a_i \otimes b_i \otimes 1, \;\;
		r_{13} = \sum\nolimits_{i} a_i \otimes 1 \otimes b_i, \;\;
		r_{23} = \sum\nolimits_{i} 1 \otimes a_i \otimes b_i,
	\end{equation*}
	where $1$ is the unit if $(A, \diamond)$ is unital or a symbol playing a similar role as the unit for the non-unital cases.
	Further define compound symbols such as $r_{12} \diamond r_{13}$ by
	\begin{equation*}
		r_{12} \diamond r_{13} = \sum\nolimits_{i, j} a_i \diamond a_j \otimes b_i \otimes b_j.
	\end{equation*}
	\item 
	Denote the standard pairing between the dual space $V^*$ and $V$ by 
	\begin{equation*}
		\langle \ ,\ \rangle: V^* \times V \to \mathbb{K}, \;\; \langle f, v \rangle := f(v), \;\; \forall f \in V^*, \; v \in V.
	\end{equation*} 
	
	\item 
	Let $V, W$ be two vector spaces and $T: V \to W$ be a linear map.
	Denote the dual map by $T^*: W^* \to V^*$, which is defined by 
	\begin{equation*}
		\langle v, T^*(w^*)\rangle = \langle T(v), w^*\rangle, \;\; \forall v \in V, w^* \in W^*.
	\end{equation*}

	\item 
	Let $A,V$ be  vector spaces.
	For a linear map $\mu: A \to \End(V)$, define a linear map $\mu^*: A \to \End(V^*)$ by $\mu^*(a) = (\mu(a))^*$, or, more explicitly
	\begin{equation*}
		\langle \mu^*(a)v^*, u\rangle = \langle v^*, \mu(a) u\rangle, \;\; \forall a \in A, u \in V, v^* \in V^*.
	\end{equation*}
	
\end{enumerate}

%%%%%%%%%%%%%%%%%%%%%%%%%%%%%%%%%%%%%%%%%%%%%%%%%%%%%%%%%%%%%%%%%%%%%%%%%%%%%%%%%
%    section  2  Lie coalgebras on tensor product of Zinbiel coalgebras and Leibniz coalgebras
%%%%%%%%%%%%%%%%%%%%%%%%%%%%%%%%%%%%%%%%%%%%%%%%%%%%%%%%%%%%%%%%%%%%%%%%%%%%%%%%%%%%%%
\section{Lie (co)algebras on tensor product of Zinbiel (co)algebras and Leibniz (co)algebras}\label{sec:zinb}
In this section, we recall the notions of Zinbiel (co)algebras and Leibniz (co)algebras, and show that there are Lie (co)algebra structures on the tensor product of Zinbiel (co)algebras and Leibniz (co)algebras.

\subsection{Zinbiel algebras and Leibniz algebras}
\begin{defi}\label{def:zin-alg}
	A {\bf (left) Zinbiel algebra} $(A, \cdot)$ is a vector space $A$ together with a binary operation $\cdot: A \otimes A \rightarrow A$ satisfying the following Zinbiel identity:
	\begin{equation*}
		a_{1}\cdot(a_{2}\cdot a_{3})=(a_{1}\cdot a_{2})\cdot a_{3}+(a_{2}\cdot a_{1})\cdot a_{3}, \;\; \forall a_{1}, a_{2}, a_{3}\in A.
	\end{equation*}
\end{defi}
A Zinbiel algebra $A$ is left commutative, i.e., $a_{1}\cdot(a_{2}\cdot a_{3})=a_{2}\cdot(a_{1}\cdot a_{3})$ for all $a_{1}, a_{2}, a_{3}\in A$. 
For the classification of low-dimensional Zinbiel algebras, see \cite{AKO,AJK,DT}. 
Below we present some examples of Zinbiel algebras.

\begin{ex}\label{ex:zin-alg1}
	\begin{enumerate}[label=(\roman*), leftmargin=1.6em]
		\item 
		$(A, \cdot)$ is a 2-dimensional Zinbiel algebra with a basis $\{e_{1}, e_{2}\}$ whose non-zero product $\cdot$ is given by $e_{1}\cdot e_{1}=e_{2}$.

		\item 
		$(A, \cdot)$ is a 3-dimensional Zinbiel algebra with a basis $\{e_{1}, e_{2}, e_{3}\}$ whose non-zero product $\cdot$ is given by $e_{1}\cdot e_{1}=e_{1}\cdot e_{2}=e_{2}\cdot e_{1}=e_{3}$.
		
		\item 
		$(\Bbbk[x], \cdot)$ is a Zinbiel algebra, where $\Bbbk[x]$ denotes the set of polynomials in one indeterminate over $\Bbbk$ and the product $\cdot$ is given by
		\begin{equation*}
			x^{m}\cdot x^{n}=\frac{1}{m+1}x^{m+n+1}, \;\; \forall m, n \in \mathbb{N}.
		\end{equation*}
	\end{enumerate}
\end{ex}

There are close connections between Zinbiel algebras and various types of algebras.
A Zinbiel algebra is equivalent to a commutative dendriform algebra. 
Consequently, if $(A, \cdot)$ is a Zinbiel algebra, then $(A, \diamond)$ is a commutative associative algebra, where $a_1 \diamond a_2 := a_1 \cdot a_2 + a_2 \cdot a_1$ for all $a_1, a_2 \in A$.
Moreover, every Zinbiel algebra under the commutator product yields a Tortkara algebra \cite{DIM,GKK}.

\begin{defi}\label{def:zrep}
	A {\bf representation} of a Zinbiel algebra $(A, \cdot)$ is a triple $(V, \kl, \kr)$, 
	where $V$ is a vector space, $\kl, \kr: A \rightarrow \End(V)$ are linear maps such that the following equations hold for all $a_{1}, a_{2}\in A$:
	\begin{align*}
		&\kl(a_{2})\kl(a_{1})=\kl(a_{2} \cdot a_{1}) + \kl(a_{1} \cdot a_{2}), \\
		&\kr(a_{1} \cdot a_{2})=\kr(a_{2})\kr(a_{1}) + \kr(a_{2}) \kl(a_{1}) = \kl(a_{1}) \kr(a_{2}).
	\end{align*}
	Two representations $(V, \kl, \kr)$ and $(V', \kl', \kr')$ are called {\bf equivalent} if there exists a linear isomorphism $f: V \rightarrow V'$ such that $f(\kl(a)(v))=\kl'(a)(f(v))$ and $f(\kr(a)(v))=\kr'(a)(f(v))$ for all $a\in A$ and $v\in V$.
\end{defi}

Let $(A, \cdot)$ be a Zinbiel algebra and $\kl, \kr: A \rightarrow \End(V)$ be linear maps. 
Define a binary operation, still denoted by $\cdot$, on $A \oplus V$ by
\begin{equation*}
	(a + u) \cdot (b + v) := a \cdot b + (\kl(a)v + \kr(b) u), \;\; \forall a, b \in A, u, v \in V.
\end{equation*}
Then $(V, \kl, \kr)$ is a representation of $(A, \cdot)$ if and only if $(A \oplus V, \cdot)$ is a  Zinbiel algebra, which is denoted by $(A \ltimes_{\kl, \kr} V, \cdot)$ and called the {\bf semi-direct product Zinbiel algebra by $(A, \cdot)$ and $(V, \kl, \kr)$}.

\begin{ex}[\cite{Wan}]
	Let $(A, \cdot)$ be a Zinbiel algebra. 
	Then $(A, \fl_{\cdot}, \fr_{\cdot})$ is a representation of $(A, \cdot)$, 
	called the {\bf regular representation}, 
	and $(A^*, \fl_{\cdot}^*+\fr_{\cdot}^*, -\fr_{\cdot}^*)$ is also a representation of $(A, \cdot)$, called the {\bf coregular representation}. 
\end{ex}

Recall that a bilinear form $\kappa$ on a Zinbiel algebra $(A, \cdot)$ is called {\bf invariant} if 
\begin{equation*}
	\kappa(a_{1}\cdot a_{2}, \; a_{3}) = \kappa(a_{2}, \; a_{1}\cdot a_{3} + a_{3}\cdot a_{1}), \;\; \forall a_{1}, a_{2}, a_{3} \in A.
\end{equation*}
A {\bf quadratic Zinbiel algebra} is a triple $(A, \cdot, \kappa)$ 
where $(A, \cdot)$ is a Zinbiel algebra and $\kappa$ is a nondegenerate skew-symmetric invariant bilinear form on $(A, \cdot)$.
If $(A, \cdot, \kappa)$ is a quadratic Zinbiel algebra, then one easily checks that $\kappa(a_{1} \cdot a_{2},\; a_{3}) = -\kappa(a_{1},\; a_{3} \cdot a_{2})$ for all $a_{1}, a_{2}, a_{3} \in A$.
Moreover, for a quadratic Zinbiel algebra $(A, \cdot, \kappa)$, the regular representation and coregular representation are equivalent.

\begin{defi}\label{def:lei-alg}
	A {\bf (left) Leibniz algebra} $(B, \circ)$ is a vector space $B$ together with a binary operation $\circ: B\otimes B\rightarrow B$ satisfying the following Leibniz identity:
	\begin{equation*}
		b_{1} \circ (b_{2} \circ b_{3})=(b_{1} \circ b_{2})\circ b_{3} + b_{2} \circ (b_{1} \circ b_{3}), \;\; \forall b_{1}, b_{2}, b_{3}\in B.
	\end{equation*}
	A {\bf $\bz$-graded Leibniz algebra} is a Leibniz
	algebra $(B, \circ)$ with a linear decomposition
	$B=\oplus_{i \in \bz }B_{i}$ such that each $B_{i}$ is finite-dimensional and 
	$B_{i} \circ B_{j}\subseteq B_{i+j}$ for all $i, j \in \bz$.
\end{defi}

\begin{ex}\label{ex:lei-alg}
	\begin{enumerate}[label=(\roman*), leftmargin=1.6em]
		\item 
		$(B, \circ)$ is a 2-dimensional Leibniz algebra with a basis $\{e_{1}, e_{2}\}$ whose non-zero product $\circ$ is given by $e_{2} \circ e_{1} = e_{1} = e_{2} \circ e_{2}$. 
		
		\item 
		$(B, \circ)$ is a 3-dimensional Leibniz algebra with a basis $\{e_{1}, e_{2}, e_{3}\}$ whose non-zero product $\circ$ is given by $e_{3} \circ e_{1}=e_{2}, \; e_{3} \circ e_{2} = e_{1}$.
	\end{enumerate}
\end{ex}

\begin{ex}\label{ex:gr-Leib}
	$(V, \diamond)$ is a 4-dimensional Leibniz algebra with a basis $\{v_{1}, v_{2}, v_{3}, v_{4}\}$ whose non-zero product $\diamond$ is given by
	\begin{equation*}
		v_{1} \diamond v_{2} = v_{1} = -v_{2} \diamond v_{1}, \;\;  
		v_{1}\diamond v_{3} = -v_{4}, \;\;
		v_{2} \diamond v_{3} = v_{3}.
	\end{equation*}
	Then we get an affine Leibniz algebra $(\widehat{V}_{4}:=
	V[\kt, \kt^{-1}] = \oplus_{i\in\bz} V_{i}, \ast)$, where $V_{i}=\Bbbk\{v_{1}\kt^{i}, v_{2}\kt^{i}, v_{3}\kt^{i}, v_{4}\kt^{i}\}$ and the product $\circ$ is given by 
	\begin{equation*}
		x\kt^{i} \circ y\kt^{j} = x \diamond y\kt^{i+j}, \;\; \forall x, y \in V, \; i, j \in \bz,
	\end{equation*}
	which is a $\bz$-graded Leibniz algebra.
\end{ex}

\begin{defi}\label{def:lrep}
	A {\bf representation} of a Leibniz algebra $(B, \circ)$ is a triple $(V, \kl, \kr)$,
	where $V$ is a vector space, $\kl, \kr: B \rightarrow \End(V)$ are linear maps such
	that the following equations hold for all $b_{1}, b_{2} \in B$:
	\begin{align*}
		&\kl(b_{1} \circ b_{2}) = \kl(b_{1})\kl(b_{2}) - \kl(b_{2})\kl(b_{1}),\\
		&\kr(b_{1})\kr(b_{2}) = \kr(b_{2} \circ b_{1}) - \kl(b_{2})\kr(b_{1}) = -\kr(b_{1})\kl(b_{2}).
	\end{align*}
	Two representations $(V, \kl, \kr)$ and $(V', \kl', \kr')$ are called {\bf equivalent} if there exists a linear isomorphism $f: V \rightarrow V'$ such that $f(\kl(b)(v))=\kl'(b)(f(v))$ and $f(\kr(b)(v))=\kr'(b)(f(v))$ for all $b \in B$ and $v \in V$.
\end{defi}

\begin{ex}
	Let $(B, \circ)$ be a Leibniz algebra. 
	Then $(B, \fl_{\circ}, \fr_{\circ})$ is a representation of $(B, \circ)$, 
	called the {\bf regular representation}, 
	and $(B^*, -\fl_{\circ}^*, \fl_{\circ}^*+\fr_{\circ}^*)$ is also a representation of $(B, \circ)$, called the {\bf coregular representation}. 
\end{ex}

A bilinear form $\omega$ on a Leibniz algebra $(B, \circ)$ is called {\bf invariant} if 
\begin{equation*}
	\omega(b_{1} \circ b_{2},\; b_{3}) = \omega(b_{1},\; b_{2} \circ b_{3} + b_{3} \circ b_{2}), \;\; \forall b_{1}, b_{2}, b_{3}\in B.
\end{equation*}
A Leibniz algebra $(B, \circ)$ with a nondegenerate skew-symmetric invariant bilinear form $\omega$ is called a {\bf quadratic Leibniz
algebra} (which is also called a skew-symmetric quadratic Leibniz algebra in \cite{BLST, TS}) and denoted by $(B, \circ, \omega)$. 
If $(B, \circ, \omega)$ is a quadratic Leibniz algebra, then $\omega(b_{1} \circ b_{2}, \; b_{3}) = -\omega(b_{2}, \; b_{1} \circ b_{3})$ for all $b_{1}, b_{2}, b_{3} \in B$.
Moreover, for a quadratic Leibniz algebra $(B, \circ, \omega)$, the regular representation and coregular representation are equivalent.

\subsection{Completed Lie coalgebras on tensor product}
Thanks to the Koszul duality between the operads of Zinbiel algebras and Leibniz algebras, the tensor product of a Zinbiel algebra and a Leibniz algebra carries a Lie algebra structure \cite{GK, LV}, as stated in the following result.
\begin{pro}\label{pro:L-Z-lie}
	Let $(A, \cdot)$ be a Zinbiel algebra and $(B, \circ)$ a Leibniz algebra.
	Define a binary operation $[-,-]$ on $A\otimes B$ by
	\begin{equation*}
		[a_{1} \otimes b_{1}, a_{2} \otimes b_{2}] = a_{1} \cdot a_{2} \otimes b_{1} \circ b_{2} - a_{2} \cdot a_{1} \otimes  b_{2} \circ b_{1}, \;\; \forall a_{1}, a_{2}\in A, \; b_{1}, b_{2}\in B.
	\end{equation*}
	Then $(A \otimes B,\; [-, -])$ is a Lie algebra, called the {\bf induced Lie algebra} from $(A, \cdot)$ and $(B, \circ)$.
\end{pro}
Dually, the tensor product of a Zinbiel coalgebra and a Leibniz coalgebra admits a Lie coalgebra structure, yielding the expected coalgebra analogues of Proposition~\ref{pro:L-Z-lie}. 
Before proceeding further, we introduce the notion of a completed tensor product to serve as the target space of more general coproducts. 
See \cite{Tak} for more details.

Let $U=\oplus_{i\in\bz}U_{i}$ and $V=\oplus_{j\in\bz}V_{j}$ be $\bz$-graded vector spaces with all $U_i$ and $V_i$ finite-dimensional.
We call the {\bf completed tensor product} of $U$ and $V$ to be the vector space
\begin{equation*}
	U\,\hat{\otimes}\,V := \prod_{i,j\in\bz}U_{i}\otimes V_{j}.
\end{equation*}
If $U$ and $V$ are finite-dimensional, then $U\,\hat{\otimes}\,V$ is just the usual tensor product $U\otimes V$. 
Moreover, a general term of $U\,\hat{\otimes}\,V$ is a possibly infinite sum $\sum_{i,j,\alpha}u_{i\alpha}\otimes v_{j\alpha}$, where $i, j\in\bz$ and $\alpha$ is in a finite index set (which might depend on $i, j$). 
With these notations, for linear maps $f: U\rightarrow U'$ and $g: V\rightarrow V'$,
define
\begin{equation*}
	f\hat{\otimes}g: U \hat{\otimes} V \rightarrow U' \hat{\otimes} V',
	\qquad \sum_{i,j,\alpha}u_{i,\alpha}\otimes v_{j, \alpha} \mapsto \sum_{i,j,\alpha} f(u_{i, \alpha})\otimes g(v_{j, \alpha}).
\end{equation*}
Moreover, a (completed) coproduct is defined as a linear map $\vartheta: V \rightarrow V \hat{\otimes} V$, and then the composite map $(\vartheta \hat{\otimes} \id)\vartheta: V \to V \hat{\otimes} V \hat{\otimes} V$ is well-defined.
Also the twisting map $\tau$ has its completion
\begin{equation*}
	\hat{\tau}: V \hat{\otimes} V\rightarrow V \hat{\otimes} V, \qquad \sum_{i,j,\alpha}u_{i, \alpha}\otimes v_{j, \alpha} \mapsto
	\sum_{i,j,\alpha} v_{j, \alpha} \otimes u_{i, \alpha}.
\end{equation*}
Finally, let $U$ be a vector space and $V = \oplus_{j \in \mathbb{Z}}V_j$ be a $\bz$-graded vector space with all $V_i$ finite-dimensional. For all $\sum_{i_1, \cdots, i_l} u_{i_1}
\otimes \cdots \otimes u_{i_l} \in U \otimes \cdots \otimes U$ and
$\sum_{j_1, \cdots, j_l, \alpha} v_{j_1\alpha} \otimes \cdots
\otimes v_{j_l\alpha} \in V \hat{\otimes} \cdots \hat{\otimes} V$,
set
\begin{equation*}
	\sum_{i_1, \cdots, i_l} u_{i_1} \otimes \cdots \otimes u_{i_l} \bullet \!\! \sum_{j_1, \cdots, j_l, \alpha} v_{j_1\alpha} \otimes \cdots \otimes v_{j_l\alpha} :=  \!\!\sum_{j_1, \cdots, j_l, \alpha}  \sum_{i_1, \cdots, i_l}  (u_{i_1} \otimes v_{j_1\alpha}) \otimes \cdots \otimes (u_{i_l} \otimes v_{j_l\alpha}) \in (U \otimes V) \hat{\otimes} \cdots \hat{\otimes} (U \otimes V).
\end{equation*}

\begin{defi}\label{de:lei-coa}
	\begin{enumerate}[leftmargin=1.6em, label=(\roman*)]
		\item 
		A {\bf completed Leibniz coalgebra} is a pair $(B, \vartheta)$, where $B=\oplus_{i\in\bz}B_{i}$ is a $\bz$-graded vector space with all $B_i$ finite-dimensional and $\vartheta: B\rightarrow B \,\hat{\otimes}\, B$ is a linear map satisfying
		\begin{equation}
			(\id\,\hat{\otimes}\,\vartheta)\vartheta=(\vartheta\,\hat{\otimes}\,\id)\vartheta +(\hat{\tau}\,\hat{\otimes}\,\id)(\id\,\hat{\otimes}\,\vartheta)\vartheta.  \label{lia1}
		\end{equation}
		When $B = B_{0}$, it is simply called a {\bf Leibniz coalgebra}.

		\item 
		A {\bf completed Lie coalgebra} is a pair $(\g, \delta)$, where $\g=\oplus_{i\in\bz}\g_{i}$ is a $\bz$-graded vector space with all $\g_i$ finite-dimensional and $\delta: \g\rightarrow\g\,\hat{\otimes}\,\g$ is a linear map satisfying Eq.~\eqref{lia1} and
		$\delta=-\hat{\tau}\,\delta$.
		When $\g = \g_{0}$, it is simply called a {\bf Lie coalgebra}.
	\end{enumerate}
\end{defi}

\begin{defi}
	A {\bf Zinbiel coalgebra} $(A, \Delta)$ is a vector space $A$ with a linear map $\Delta: A\rightarrow A\otimes A$ such that $(\id\otimes\Delta)\Delta = (\Delta\otimes\id)\Delta+(\tau\otimes\id) (\Delta\otimes\id)\Delta$;
\end{defi}

Now we give the dual version of Proposition~\ref{pro:L-Z-lie}.

\begin{pro}\label{pro:L-Z-colie}
	Let $(A, \Delta)$ be a Zinbiel coalgebra and $(B, \vartheta)$ be a completed Leibniz coalgebra. 
	Define a linear map $\delta: A\otimes B\rightarrow(A\otimes B) \,\hat{\otimes}\,(A\otimes B)$ by
	\begin{equation}
		\delta(a\otimes b)=(\id-\hat{\tau})\big(\Delta(a)\bullet\vartheta(b)\big),  \;\; \forall a \in A, b \in B.\label{colie}
	\end{equation}
	Then $(A\otimes B,\; \delta)$ is a completed Lie coalgebra.
\end{pro}
\begin{proof}
	Clearly, $\delta=-\hat{\tau}\,\delta$.
	For all $a \in A$ and $b \in B$, we have
	\begin{align*}
		&\;((\id\hat{\otimes}\delta)\delta-(\hat{\tau}\hat{\otimes}\id)(\id\hat{\otimes}\delta)\delta
		-(\delta\hat{\otimes}\id)\delta )(a\otimes b)\\
		=&\; (\id\otimes\Delta)\Delta(a) \bullet (\id\hat{\otimes}\vartheta)\vartheta(b)
		- (\id\otimes\tau)(\id\otimes\Delta)\Delta(a)\bullet
		(\id\hat{\otimes}\hat{\tau})(\id\hat{\otimes}\vartheta)\vartheta(b) \\[-1mm]
		&\quad-(\id\otimes\Delta)\tau\Delta(a) \bullet (\id\hat{\otimes}\vartheta)\hat{\tau}\vartheta(b) + (\id\otimes\tau)(\id\otimes\Delta)\tau\Delta(a) \bullet
		(\id\hat{\otimes}\hat{\tau})(\id\hat{\otimes}\vartheta)\hat{\tau}\vartheta(b) \\[-1mm]
		&\quad - (\tau\otimes\id)(\id\otimes\Delta)\Delta(a) \bullet
		 (\hat{\tau}\hat{\otimes}\id)(\id\hat{\otimes}\vartheta)\vartheta(b) \\[-1mm]
		&\quad+(\tau\otimes\id)(\id\otimes\tau)(\id\otimes\Delta)\Delta(a)\bullet
		(\hat{\tau}\hat{\otimes}\id)(\id\hat{\otimes}\hat{\tau})(\id\hat{\otimes}\vartheta)\vartheta(b) \\[-1mm]
		&\quad+  (\tau\otimes\id)(\id\otimes\Delta)\tau\Delta(a)\bullet (\hat{\tau}\hat{\otimes}\id)(\id\hat{\otimes}\vartheta)\hat{\tau}\vartheta(b)\\[-1mm]
		&\quad-(\tau\otimes\id)(\id\otimes\tau)(\id\otimes\Delta)\tau\Delta(a)\bullet
		(\hat{\tau}\hat{\otimes}\id)(\id\hat{\otimes}\hat{\tau})(\id\hat{\otimes}\vartheta)\hat{\tau}\vartheta(b)\\[-1mm]
		&\quad-(\Delta\otimes\id)\Delta(a) \bullet
		(\vartheta\hat{\otimes}\id)\vartheta(b)
		+(\tau\otimes\id)(\Delta\otimes\id)\Delta(a) \bullet
		(\hat{\tau}\hat{\otimes}\id)(\vartheta\hat{\otimes}\id)\vartheta(b)\\[-1mm]
		&\quad+(\Delta\otimes\id)\tau\Delta(a)\bullet (\vartheta\hat{\otimes}\id)\hat{\tau}\vartheta(b) -(\tau\otimes\id)(\Delta\otimes\id)\tau\Delta(a)\bullet
		(\hat{\tau}\hat{\otimes}\id)(\vartheta\hat{\otimes}\id)\hat{\tau}\vartheta(b).
	\end{align*}
	Note that
	\begin{align*}
		(\id\,\hat{\otimes}\,\vartheta)\vartheta - (\hat{\tau}\hat{\otimes}\id)(\id\,\hat{\otimes}\,\vartheta)\vartheta = (\vartheta\hat{\otimes}\id)\vartheta = -(\hat{\tau}\hat{\otimes} \id)(\vartheta\hat{\otimes}\id)\vartheta, \\
		(\Delta\otimes\id)\Delta+(\tau\otimes\id) (\Delta\otimes\id)\Delta = (\id\otimes\Delta)\Delta = (\tau\otimes\id)(\id\otimes\Delta)\Delta.
	\end{align*}
	Therefore, we have
	\begin{align*}
		&\;((\id\,\hat{\otimes}\,\delta)\delta
		-(\hat{\tau}\,\hat{\otimes}\,\id)(\id\,\hat{\otimes}\,\delta)\delta
		-(\delta\,\hat{\otimes}\,\id)\delta )(a\otimes b)\\[-1mm]
		=&\;\Big((\id\otimes\Delta)\Delta-(\Delta\otimes\id)\Delta - 
		(\tau\otimes\id)(\Delta\otimes\id)\Delta\Big)(a)
		\bullet(\id\hat{\otimes}\vartheta)\vartheta(b)\\[-1mm]
		&\quad+\Big((\Delta\otimes\id)\Delta+(\tau\otimes\id)(\Delta\otimes\id)\Delta
		-(\tau\otimes\id)(\id\otimes\Delta)\Delta\Big)(a)
		\bullet(\hat{\tau}\hat{\otimes}\id)(\id\hat{\otimes}\vartheta)\vartheta(b)\\[-1mm]
		&\quad+\Big((\tau\otimes\id)(\id\otimes\Delta)\tau\Delta
		+(\tau\otimes\id)(\id\otimes\tau)(\id\otimes\Delta)\tau\Delta
		-(\id\otimes\tau)(\id\otimes\Delta)\Delta\Big)(a)\\[-1mm]
		&\qquad\qquad\qquad\qquad\qquad\qquad \bullet(\id\hat{\otimes}\hat{\tau})(\id\hat{\otimes}\vartheta)
		\vartheta(b)\\[-1mm]
		&\quad+\Big((\Delta\otimes\id)\tau\Delta-(\tau\otimes\id)(\id\otimes\Delta)\tau\Delta
		-(\tau\otimes\id)(\id\otimes\tau)(\id\otimes\Delta)\tau\Delta\Big)(a)\\[-1mm]
		&\qquad\qquad\qquad\qquad\qquad\qquad \bullet(\id\hat{\otimes}\hat{\tau})(\hat{\tau}\hat{\otimes}\id)
		(\id\hat{\otimes}\hat{\vartheta})\vartheta(b)\\[-1mm]
		&\quad+\Big((\tau\otimes\id)(\id\otimes\tau)(\id\otimes\Delta)\Delta
		-(\id\otimes\Delta)\tau\Delta-(\id\otimes\tau)(\id\otimes\Delta)\tau\Delta\Big)(a)\\[-1mm]
		&\qquad\qquad\qquad\qquad\qquad\qquad \bullet(\hat{\tau}\hat{\otimes}\id)(\id\hat{\otimes}\hat{\tau})
		(\id\hat{\otimes}\vartheta)\vartheta(b)\\[-1mm]
		&\quad+\Big((\id\otimes\Delta)\tau\Delta+(\id\otimes\tau)(\id\otimes\Delta)\tau\Delta
		-(\tau\otimes\id)(\Delta\otimes\id)\tau\Delta\Big)(a)\\[-1mm]
		&\qquad\qquad\qquad\qquad\qquad\qquad \bullet(\id\hat{\otimes}\hat{\tau})(\hat{\tau}\hat{\otimes}\id)
		(\id\hat{\otimes}\hat{\tau})(\id\hat{\otimes}\vartheta)\vartheta(b)\\
		=&\; 0.
	\end{align*}
	Therefore, $(A\otimes B, \delta)$ is a
	completed Lie coalgebra. 
\end{proof}

%\smallskip
%%%%%%%%%%%%%%%%%%%%%%%%%%%%%%%%%%%%%%%%%%%%%%%%%%%%%%%%%%%%%%%%%%%%%%%%%%%%%%%%%
%    section  3  Leibniz Bialgebras
%%%%%%%%%%%%%%%%%%%%%%%%%%%%%%%%%%%%%%%%%%%%%%%%%%%%%%%%%%%%%%%%%%%%%%%%%%%%%%%%%%%%%%
\section{Lie bialgebras from Leibniz bialgebras} \label{sec:Leib-Liebia}
In this section, we construct a Lie bialgebra on the tensor product of a Leibniz bialgebra and a quadratic Zinbiel algebra, and show that this Lie bialgebra is quasi-triangular (resp. triangular, factorizable) whenever the Leibniz bialgebra is quasi-triangular (resp. triangular, factorizable).

%%%%%%%%%%%%%%%%%%%%%%%%%%%%%%%%%%%%%%%%%%%%%%%%%%%%%%%%%%%%%%%%%%%%%%%%%%%
\subsection{Lie bialgebra structures on the tensor products}\label{ssec:Li-Lib}

Let $\omega$ be a nondegenerate bilinear form on a vector space $V$. 
For all $l \geq 2$, the bilinear form $\widetilde{\omega}_l: \overbrace{(V \otimes \cdots \otimes V)}^{l-\text{fold}} \otimes \overbrace{(V \otimes \cdots \otimes V)}^{l-\text{fold}} \to \mathbf{k}$ defined by
\begin{equation*}
	\widetilde{\omega}_l(a_{1}\otimes a_{2}\otimes\cdots\otimes a_{l}, \, a'_{1}\otimes a'_{2} \otimes\cdots\otimes a'_{l})=\prod_{i=1}^{l}\omega(a_{i}, a'_{i})
\end{equation*}
is {\bf nondegenerate}. 
For brevity, we will suppress the index $l$ without ambiguity.

\begin{lem}\label{lem:qZ-dual}
	Let $(A, \cdot, \kappa)$ be a quadratic Zinbiel algebra. 
	Define a linear map $\Delta: A\rightarrow A \otimes A$ by
	\begin{equation}
		\widetilde{\kappa}(\Delta (a_{1}),\; a_{2}\otimes a_{3})=-\kappa(a_{1},\; a_{2}\cdot a_{3}), \;\; \forall a_{1}, a_{2}, a_{3}\in A. \label{eq:q2zc}
	\end{equation}
	Then $(A, \Delta)$ is a Zinbiel coalgebra.
\end{lem}
\begin{proof}
	For all $a, a_{1}, a_{2}, a_{3}\in A$, we have
	\begin{align*}
	&\; \widetilde{\kappa}((\Delta\otimes\id)\Delta(a)
	+(\tau\otimes\id)(\Delta\otimes\id)\Delta(a) - (\id\otimes\Delta)\Delta(a),\; a_{1}\otimes a_{2}\otimes a_{3})\\
	=&\;\kappa\big(a,\; (a_{1}\cdot a_{2})\cdot a_{3}+(a_{2}\cdot a_{1})\cdot a_{3}- a_{1} \cdot (a_{2} \cdot a_{3})\big)
	=0.
	\end{align*}
	That is, $(\id\otimes\Delta)\Delta=(\Delta\otimes\id)
	\Delta+(\tau\otimes\id)(\Delta\otimes\id)\Delta$ and therefore $(A, \Delta)$ is a Zinbiel coalgebra.
\end{proof}

\begin{ex}\label{ex:qu-zib}
	Let $(A, \cdot)$ be a $4$-dimensional Zinbiel algebra with a basis $\{e_{1}, e_{2}, e_{3}, e_{4}\}$ whose non-zero product $\cdot$ is given by 
	\begin{equation*}
		e_{1}\cdot e_{1}=e_{2}, \;
		e_{4}\cdot e_{4}=e_{3}, \;
		e_{1}\cdot e_{4} =2e_{3}-e_{2}, \;
		e_{4}\cdot e_{1}=2e_{2}-e_{3}.
	\end{equation*}
	Define a bilinear form $\kappa$ on $(A, \cdot)$ by setting
	\begin{equation*}
		\kappa(e_{3}, e_{1}) = \kappa(e_{4}, e_{2}) = 1 = -\kappa(e_{2}, e_{4}) = -\kappa(e_{1}, e_{3}),
	\end{equation*}
	and all remaining basis pairings zero. 
	Then $(A, \cdot, \kappa)$ is a quadratic Zinbiel algebra.
	Furthermore, by Lemma~\ref{lem:qZ-dual}, we obtain a Zinbiel coalgebra $(A, \Delta)$, where $\Delta$ is given by
	\begin{equation*}
		\Delta(e_{1}) = e_{2}\otimes e_{2} + e_{2}\otimes e_{3} - 2e_{3}\otimes e_{2}, \; 
		\Delta(e_{2}) =\Delta(e_{3})=0, \;
		\Delta(e_{4}) = - e_{3}\otimes e_{3} + 2e_{2}\otimes e_{3} - e_{3}\otimes e_{2}.
	\end{equation*}
\end{ex}

The following notions and theorem present a construction of Lie bialgebras from a Leibniz bialgebra and a quadratic Zinbiel algebra, lifting Propositions~\ref{pro:L-Z-lie} and
\ref{pro:L-Z-colie} to the level of bialgebras.
\begin{defi}
	A {\bf Leibniz bialgebra} is a triple $(B, \circ, \vartheta)$,
	where $(B, \circ)$ is a Leibniz algebra, $(B, \vartheta)$ is a Leibniz coalgebra and the following equations hold:
	\begin{align*}
		&\vartheta(b_{1} \circ b_{2})=(\id\otimes\fr_{\circ}(b_{2})-(\fl_{\circ}+\fr_{\circ})(b_{2})\otimes\id)
		(\id+\tau) \vartheta(b_{1}) + (\id\otimes\fl_{\circ}(b_{1})
		+\fl_{\circ}(b_{1})\otimes\id)\vartheta(b_{2}), \\
		&\tau(\fr_{\circ}(b_{2})\otimes\id)\vartheta(b_{1})
		=(\fr_{\circ}(b_{1})\otimes\id)\vartheta(b_{2}), \;\; \forall b_{1}, b_{2}\in B.
	\end{align*}
\end{defi}

\begin{defi}\label{def:com-liebi}
	A {\bf completed Lie bialgebra} is a triple $(\g=\oplus_{i\in\bz}\g_{i}, [-,-], \delta)$ such that $(\g, [-, -])$ is a Lie algebra, $(\g, \delta)$ is a completed Lie coalgebra, and the following compatibility condition holds:
	\begin{equation}
		\delta([g_{1}, g_{2}])=\big(\ad(g_{1})\,\hat{\otimes}\,\id +\id\,\hat{\otimes}\,\ad(g_{1})\big)\delta(g_{2})-\big(\ad(g_{2})\,\hat{\otimes}\,\id
		+\id\,\hat{\otimes}\,\ad(g_{2})\big)\delta(g_{1}), \;\; \forall g_{1}, g_{2}\in\g.   \label{comlie}
	\end{equation}
	When $\g = \g_{0}$, it is simply called a {\bf Lie bialgebra}.
\end{defi}

\begin{thm}\label{thm:liebia-LZ}
	Let $(B, \circ, \vartheta)$ be a Leibniz bialgebra, $(A, \cdot, \kappa)$
	be a quadratic Zinbiel algebra and $(\g:=A \otimes B, [-,-])$ be the induced Lie
	algebra from $(A, \cdot)$ and $(B, \circ)$. 
	Let $\Delta: A \rightarrow A \otimes A$ be the linear map defined by Eq.~\eqref{eq:q2zc} and $\delta: \g \rightarrow \g \otimes \g$ defined by Eq.~\eqref{colie}.
	Then $(\g, [-,-], \delta)$ is a Lie bialgebra, which is called the
	{\bf Lie bialgebra induced from $(B, \circ, \vartheta)$ by $(A, \cdot, \kappa)$}.
\end{thm}
\begin{proof}
	By Lemma \ref{lem:qZ-dual} and Proposition \ref{pro:L-Z-colie}, $(\g, \delta)$ is a Lie coalgebra. 
	Thus, it is sufficient to show that Eq.~\eqref{comlie} holds.
	For all $a, c\in A$ and $b, d\in B$, using the Sweedler notation, we have
	\begin{align*}
	&\;(\ad(a\otimes b)\otimes\id+\id\otimes\ad(a\otimes b))\delta(c\otimes d)
	-(\ad(c\otimes d)\otimes\id+\id\otimes\ad(c\otimes d))\delta(a\otimes b)\\
	=&\;(\id-\tau)\bigg(\sum_{(c)}\sum_{(d)}\Big(
	(a\cdot c_{(1)}\otimes c_{(2)})\bullet(b\circ d_{(1)}\otimes d_{(2)})
	-(c_{(1)}\cdot a\otimes c_{(2)})\bullet(d_{(1)}\circ b\otimes d_{(2)})\\[-4mm]
	&\qquad\qquad\qquad\qquad
	+(c_{(1)}\otimes a\cdot c_{(2)} )\bullet(d_{(1)}\otimes b\circ d_{(2)}) 
	-(c_{(1)}\otimes c_{(2)}\cdot a) \bullet(d_{(1)}\otimes d_{(2)}\circ b )\Big)\\
	&\quad-\sum_{(a)}\sum_{(b)}\Big(
	( c\cdot a_{(1)} \otimes a_{(2)})\bullet( d\circ b_{(1)} \otimes b_{(2)})
	-( a_{(1)}\cdot c \otimes a_{(2)})\bullet( b_{(1)}\circ d \otimes b_{(2)})\\[-4mm]
	&\qquad\qquad\quad
	+(a_{(1)}\otimes c\cdot a_{(2)} )\bullet(b_{(1)}\otimes d\circ b_{(2)}) 
	-(a_{(1)}\otimes a_{(2)}\cdot c) \bullet(b_{(1)}\otimes b_{(2)}\circ d )\Big)\bigg).
	\end{align*}
	For all $e, f\in A$, we have
	\begin{align*}
	&\;\widetilde{\kappa}(\sum_{(c)}c_{(1)}\otimes a\cdot c_{(2)}-\Delta(a\cdot c)
	+\tau\Delta(c \cdot a),\; e\otimes f) \\[-1mm]
	=&\; -\kappa(c,\; e\cdot(a\cdot f+f\cdot a)) - \kappa(c,\; a\cdot(e\cdot f) + (e\cdot f)\cdot a) - \kappa(c,\; (f\cdot e)\cdot a)\\
	=&\; 0.
	\end{align*}
	Since $\widetilde{\kappa}$ is nondegenerate, we obtain
	\begin{equation*}
		\sum_{(c)}c_{(1)}\otimes a\cdot c_{(2)}
		=\Delta(a\cdot c)-\tau\Delta(c\cdot a).
	\end{equation*}
	Similarly, we have 
	\begin{align*}
		&\sum_{(c)} a\cdot c_{(1)} \otimes c_{(2)}
		=\Delta(a\cdot c)+\Delta(c\cdot a)
		=\sum_{(a)} c\cdot a_{(1)} \otimes a_{(2)}, \\
		&\sum_{(c)}c_{(1)}\otimes c_{(2)}\cdot a = \Delta(c\cdot a)
		+\tau \Delta(c\cdot a), \\
		&\sum_{(a)}a_{(1)}\otimes a_{(2)}\cdot c
		=\Delta(a\cdot c)+\tau \Delta(a\cdot c), \\
		&\sum_{(a)}a_{(1)}\otimes c\cdot a_{(2)} =\Delta(c\cdot a) - \tau\Delta(a\cdot c), \\
		&\sum_{(a)} a_{(1)}\cdot c \otimes a_{(2)} = - \Delta(a\cdot c) - \Delta(c\cdot a) - \tau (\Delta(a\cdot c) + \Delta(c\cdot a) + \sum_{(c)} c_{(1)} \cdot a \otimes c_{(2)}).
	\end{align*}
	Furthermore, noting that $(B, \circ, \vartheta)$ is a Leibniz bialgebra, we have
	\begin{align*}
	&\;(\ad(a\otimes b)\otimes\id+\id\otimes\ad(a\otimes b))\delta(c\otimes d) - (\ad(c\otimes d)\otimes\id+\id\otimes\ad(c\otimes d))\delta(a\otimes b)\\
	=&\;(\id-\tau)\bigg(-\tau\Delta(a\cdot c)\bullet \Big(\sum_{(b)} b_{(1)} \circ d \otimes b_{(2)}-b_{(1)}\otimes d \circ b_{(2)} 
	-b_{(1)}\otimes b_{(2)}  \circ d \Big)\\[-2mm]
	&\quad+\Delta(a\cdot c)\bullet\Big(\sum_{(d)}\big( b \circ d_{(1)} \otimes d_{(2)}
	+d_{(1)}\otimes b \circ d_{(2)} \big)\\[-4mm]
	&\qquad\qquad\qquad\qquad-\sum_{(b)}\big( d \circ b_{(1)} \otimes b_{(2)}
	- b_{(1)} \circ d \otimes b_{(2)}+b_{(1)}\otimes b_{(2)} \circ d \big)\Big)\\[-2mm]
	&\quad-\tau\Delta(c\cdot a)\bullet\Big(\sum_{(d)}\big(
	d_{(1)}\otimes d_{(2)} \circ b +d_{(1)}\otimes b \circ d_{(2)} \big)
	+\sum_{(b)} b_{(1)} \circ d \otimes b_{(2)}\Big)\\[-2mm]
	&\quad+\Delta(c\cdot a)\bullet\Big(\sum_{(d)}\big( b \circ d_{(1)} \otimes d_{(2)}
	-d_{(1)}\otimes d_{(2)} \circ b \big)\\[-4mm]
	&\qquad\qquad\qquad\qquad-\sum_{(b)}\big( d \circ b_{(1)} \otimes b_{(2)}
	- b_{(1)} \circ d \otimes b_{(2)}-b_{(1)}\otimes d \circ b_{(2)} \big)\Big)\\[-2mm]
	&\quad-\Big(\sum_{(c)} c_{(1)}\cdot a \otimes c_{(2)}\Big)\bullet
	\Big(\sum_{(d)} d_{(1)} \circ b \otimes d_{(2)}\Big)
	-\Big(\sum_{(c)}c_{(2)}\otimes c_{(1)}\cdot a \Big)\bullet
	\Big(\sum_{(b)}b_{(1)} \circ d\otimes b_{(2)}\bigg)\\[-2mm]
	%%%%%%%%%%%%%%%%%%%%%%%%%%%%%%
	=&\;(\id-\tau)\bigg(\Delta(a\cdot c)\bullet
	\Big(\sum_{(b)}b_{(2)}\otimes b_{(1)} \circ d - d \circ b_{(2)} \otimes b_{(1)}
	- b_{(2)} \circ d \otimes b_{(1)}\Big)\\[-2mm]
	&\quad+\Delta(a\cdot c)\bullet\Big(\sum_{(d)}\big( b \circ d_{(1)} \otimes d_{(2)}
	+d_{(1)}\otimes b \circ d_{(2)} \big)\\[-4mm]
	&\qquad\qquad\qquad\qquad-\sum_{(b)}\big( d \circ b_{(1)} \otimes b_{(2)}
	- b_{(1)} \circ d \otimes b_{(2)}+b_{(1)}\otimes b_{(2)} \circ d \big)\Big)\\[-2mm]
	&\quad+\Delta(c\cdot a)\bullet\Big(\sum_{(d)}\big(
	 d_{(2)} \circ b \otimes d_{(1)}+ b \circ d_{(2)} \otimes d_{(1)}\big)
	+\sum_{(b)}b_{(2)}\otimes b_{(1)} \circ d\Big)\\[-2mm]
	&\quad+\Delta(c\cdot a)\bullet\Big(\sum_{(d)}\big(b \circ d_{(1)}\otimes d_{(2)}
	-d_{(1)}\otimes d_{(2)} \circ b \big)\\[-4mm]
	&\qquad\qquad\qquad\qquad-\sum_{(b)}\big(d \circ b_{(1)}\otimes b_{(2)}
	-b_{(1)} \circ d \otimes b_{(2)}-b_{(1)}\otimes d \circ b_{(2)} \big)\Big)\\[-2mm]
	&\quad+\Big(\sum_{(c)} c_{(1)}\cdot a \otimes c_{(2)}\Big)\bullet
	\Big(\sum_{(b)}b_{(2)}\otimes b_{(1)} \circ d 
	-\sum_{(d)} d_{(1)} \circ b \otimes d_{(2)}\Big)\bigg)\\[-2mm]
	=&\; (\id - \tau)\Big(\Delta(a\cdot c)\bullet\vartheta(b \circ d)
	-\Delta(c\cdot a)\bullet\vartheta(d \circ b)\Big) \\
	=&\; \delta([a\otimes b,\; c\otimes d]),
	\end{align*}
	Therefore, $(\g, [-,-], \delta)$ is a Lie bialgebra.
\end{proof}

%%%%%%%%%%%%%%%%%%%%%%%%%%%%%%%%%%%%%%%%%%%%%%%%%%%%%%%%%%%%%%%%%%%%%%%%%%%%%%%%%%%%%
\subsection{Manin triples of Lie algebras from Manin triples of Leibniz algebras}\label{ssec:manin}
Note that a Lie bialgebra is characterized by a Manin triple of Lie algebras. 
Hence we give the correspondence between Leibniz bialgebras and Lie bialgebras in terms of Manin triples as follows.

Recall that a bilinear form $\mathcal{B}$ on a Lie
algebra $(\g, [-,-])$ is called {\bf invariant} if
\begin{equation*}
	\mathcal{B}([g_{1}, g_{2}], g_{3})=\mathcal{B}(g_{1}, [g_{2}, g_{3}]), \;\; \forall g_{1}, g_{2}, g_{3}\in\g.
\end{equation*}

\begin{defi}\cite{CP}\label{def:manin-Lie}
	Let $(\g, [-,-])$, $(\g_{1}, [-,-]_{1})$ and $(\g_{2}, [-,-]_{2})$ be three Lie algebras.
	If there is a nondegenerate invariant symmetric bilinear form $\mathcal{B}$ on $(\g, [-,-])$ such that
	\begin{itemize}
		\item[$(i)$] 
		$(\g_{1}, [-,-]_{1})$ and $(\g_{2}, [-,-]_{2})$ are Lie subalgebras of $(\g, [-,-])$ and $\g=\g_{1}\oplus\g_{2}$ as a direct sum of vector spaces;
		
		\item[$(ii)$] 
		$(\g_{1}, [-,-]_{1})$ and $(\g_{2}, [-,-]_{2})$ are isotropic with respect to $\mathcal{B}$, that is, $\mathcal{B}(\g_{i}, \g_{i})=0$ for $i=1, 2$,
	\end{itemize}
	then the triple $(\g, \g_{1}, \g_{2})$ is called a {\bf Manin triple of Lie algebras} associated with $\mathcal{B}$.
	Two Manin triples $(\g, \g_{1}, \g_{2})$ and $(\g', \g'_{1}, \g'_{2})$ associated with $\mathfrak{B}$ and $\mathfrak{B}^\prime$ respectively are called {\bf isomorphic} if there exists an isomorphism $\varphi: \g\rightarrow\g'$ of Lie algebras such that $\varphi(\g_{i})=\g'_{i}$ for $i=1,2$ and $\mathcal{B}'(\varphi(g_{1}), \varphi(g_{2}))=\mathcal{B}(g_{1}, g_{2})$ for all $g_{1}, g_{2}\in\g$.
\end{defi}

\begin{pro}[\cite{CP}]\label{pro:man-Lie}
	Consider a Lie algebra $(\g, [-, -])$ and a Lie coalgebra $(\g, \delta)$, such that the linear dual $\delta^*: \g^* \otimes \g^* \to \g^*$ defines a Lie algebra $(\g^*,\{-, -\})$.
	Then the triple $(L, [-, -], \delta)$ is a Lie bialgebra if and only if $(L \oplus L^*, L, L^*)$ is a Manin triple of Lie algebras associated with the bilinear form defined by
	\begin{equation}\label{eq:smp}
		\mathfrak{B}_d((g_{1}, \xi_{1}),\; (g_{2}, \xi_{2})) = \langle\xi_{1}, g_{2} \rangle + \langle\xi_{2}, g_{1}\rangle, \;\; \forall g_{1}, g_{2}\in\g, \xi_{1}, \xi_{2}\in\g^{\ast}.
	\end{equation}
\end{pro}

\begin{defi}[\cite{TS}]\label{def:manin-Leib}
	Let $(B, \circ)$, $(B_{1}, \circ_{1})$ and $(B_{2}, \circ_{2})$ be three Leibniz algebras.
	If $(B, \circ, \omega)$ is a quadratic Zinbiel algebra and
	\begin{itemize}
		\item[$(i)$] 
			$(B_{1}, \circ_{1})$ and $(B_{2}, \circ_{2})$ are Leibniz subalgebras of $(B, \circ)$ and $B=B_{1}\oplus B_{2}$ as a direct sum of vector spaces;
		
		\item[$(ii)$] 
		$(B_{1}, \circ_{1})$ and $(B_{2}, \circ_{2})$ are isotropic with respect to $\omega$,
	\end{itemize}
	then the triple $(B, B_{1}, B_{2})$ is called a {\bf Manin triple of Leibniz algebras} associated with $\omega$.
\end{defi}

\begin{pro}[\cite{TS}]\label{pro:man-Leib}
Let $(B, \circ)$ be a Leibniz algebra and $(B, \vartheta)$ be a Leibniz coalgebra.
Then the triple $(B, \circ, \vartheta)$ is a Leibniz bialgebra if and only if $(B\oplus
B^{\ast}, B, B^{\ast})$ is a Manin triple of Leibniz algebras associated with the
bilinear form $\omega_{d}$ on $B \oplus B^{\ast}$ defined by 
\begin{equation}
	\omega_{d}((b_{1},
	\eta_{1}),\; (b_{2}, \eta_{2}))=\langle\eta_{1}, b_{2}\rangle-\langle\eta_{2}, b_{1}\rangle, \;\; \forall b_1, b_2 \in B, \eta_{1}, \eta_{2} \in B^{\ast}. \label{eq:bleib}
\end{equation}
\end{pro}

\begin{pro}\label{pro:bilin}
	Let $(B, \circ)$ be a Leibniz algebra, $(A, \cdot, \kappa)$ be a quadratic Zinbiel algebra and $(A\otimes B, [-,-])$ be the induced Lie algebra. 
	Suppose that $\omega: B\otimes B\rightarrow B$ is a nondegenerate invariant skew-symmetric bilinear form on $(B, \circ)$.  
	Define a bilinear form $\mathcal{B}$ on $A \otimes B$ by
	\begin{equation*}
		\mathcal{B}(a_{1}\otimes b_{1},\; a_{2}\otimes b_{2})
		=\kappa(a_{1},a_{2})\omega(b_{1}, b_{2}), \;\; \forall a_{1}, a_{2}\in A, \; b_{1}, b_{2}\in B.
	\end{equation*}
	Then $\mathcal{B}$ is a nondegenerate invariant
	symmetric bilinear form on $(A\otimes B, [-,-])$.
\end{pro}
\begin{proof}
	It is a straightforward check.
\end{proof}

\begin{lem}\label{lem:man1}
	Let $(A, \cdot, \kappa)$ be a quadratic Zinbiel algebra. 
	Let $(B, \circ)$ and $(B^{\ast}, \circ')$ be Leibniz algebras, and $(A \otimes B, [-, -]_{A \otimes B})$ and $(A \otimes B^{\ast}, [-, -]_{A \otimes B^{\ast}})$ be the induced Lie algebras from $(A, \cdot)$ and $(B, \circ)$ as well as  $(A, \cdot)$ and $(B, \circ')$ respectively.
	If $(B\oplus B^{\ast}, B, B^{\ast})$ is a Manin triple of Leibniz algebras associated with $\omega_{d}$ defined by Eq.~\eqref{eq:bleib}, 
	then $((A\otimes B)\oplus(A\otimes B^{\ast}),\; A\otimes B,\; A\otimes B^{\ast})$ is a Manin triple of Lie algebras associated with the bilinear form defined by
	\begin{equation}
		\mathcal{B}((a_{1}\otimes b_{1}, a'_{1}\otimes\xi_{1}),\; (a_{2}\otimes b_{2}, a'_{2}\otimes\xi_{2}))=-\kappa(a_{1}, a'_{2})\langle\xi_{2}, b_{1}\rangle+\kappa(a'_{1}, a_{2}) \langle\xi_{1}, b_{2}\rangle,        \label{Mani1}
	\end{equation}
for all $a_{1}, a_{2}, a'_{1}, a'_{2}\in A$, $b_{1}, b_{2}\in B$ and $\xi_{1},
\xi_{2}\in B^{\ast}$.
\end{lem}
\begin{proof}
	It follows from Proposition~\ref{pro:bilin}
\end{proof}

With the same assumptions in Lemma~\ref{lem:man1}, due to the nondegeneracy of the bilinear form $\kappa$, through the linear isomorphism $\varphi: A\rightarrow A^\ast$ given by $\langle\varphi(a_{1}),\; a_{2}\rangle=\kappa(a_{1}, a_{2})$ for all $a_{1}, a_{2}\in A$, a Zinbiel algebra structure
$(A^{\ast}, \cdot')$ on $A^{\ast}$ is obtained by transporting of structure:
\begin{equation*}
	\xi\cdot'\eta=\varphi(\varphi^{-1}(\xi)\cdot\varphi^{-1}(\eta)), \;\; \forall \xi, \eta\in A^{\ast}.
\end{equation*}
Then we obtain the induced Lie algebra
$(A^* \otimes B^*, [-, -]_{A^* \otimes B^*})$ from $(A^*, \cdot')$ and $(B^*, \circ')$. 
Further, we can define a linear map $f: A \otimes (B \oplus B^*) \to (A \otimes B) \oplus
(A^* \otimes B^*)$ by
\begin{equation*}
	a_1 \otimes b_1 + a_2 \otimes \xi_2  \mapsto a_1 \otimes b_1 + \varphi(a_2) \otimes \xi_2, \;\; \forall a_1, a_2 \in A, b_1 \in B, \xi_2 \in B^*.
\end{equation*}
Evidently, $f$ is a linear isomorphism.
Thus we can transport the Lie algebra structure on $A \otimes (B \oplus B^*)$ to $(A \otimes B) \oplus (A^* \otimes B^*)$ by
\begin{equation*}
	[a_1 \otimes b_1 + \xi_1 \otimes \eta_1, a_2 \otimes b_2 + \xi_2 \otimes \eta_2]^\prime := f[a_1 \otimes b_1 + \varphi^{-1}(\xi_1) \otimes \eta_1, a_2 \otimes b_2 + \varphi^{-1}(\xi_2) \otimes \eta_2],
\end{equation*}
for all $a_{1}, a_{2}\in A$, $\xi_{1}, \xi_{2}\in A^{\ast}$, $b_{1}, b_{2}\in B$ and $\eta_{1}, \eta_{2}\in B^{\ast}$. 
Moreover, the restrictions of the Lie bracket $[-,-]^\prime$ on $A\otimes B$ and $A^* \otimes B^*$ coincide with $[-, -]_{A \otimes B}$ and $[-, -]_{A^* \otimes B^*}$ respectively. Hence, we arrive at the following conclusion
immediately.

\begin{lem}\label{lem:man2}
	With the notations above,
	$(\g \oplus \g^{\ast}, \g = A \otimes B, \g^* = A^{\ast} \otimes B^{\ast})$ is a Manin triple of Lie algebras associated with the bilinear form defined by Eq.~\eqref{eq:smp}.
	Moreover, this Manin triple is isomorphic to the Manin triple $((A\otimes B)\oplus(A\otimes B^{\ast}),\; A\otimes B,\; A\otimes B^{\ast})$ given in Lemma~\ref{lem:man1}.
\end{lem}

Lemmas~\ref{lem:man1} and \ref{lem:man2} yield the following proposition and commutative diagram.

\begin{pro}\label{pro:comm}
Let $(B, \circ, \vartheta)$ be a Leibniz bialgebra, $(A, \cdot, \kappa)$
be a quadratic Zinbiel algebra and $(A\otimes B, [-,-], \delta)$ be the Lie bialgebra
induced from $(B, \circ, \vartheta)$ by $(A, \cdot, \kappa)$.
Then $(A\otimes B, [-,-], \delta)$ coincides with the one obtained from the Manin triple $((A\otimes B)\oplus(A^{\ast}\otimes B^{\ast}),\; A\otimes B,\;
A^{\ast}\otimes B^{\ast})$ of Lie algebras given in Lemma \ref{lem:man2}, leading to the commutative diagram:
\begin{equation*}
	\xymatrix@C=3cm@R=0.6cm{
		\txt{$(B, \circ, \vartheta)$ \\ {\tiny a Leibniz bialgebra}}
		\ar[d]_{\rm Thm. ~\ref{thm:liebia-LZ}}\ar@{<->}[r]^{\rm Pro. ~\ref{pro:man-Leib}} &
		\txt{$(B\oplus B^{\ast}, B, B^{\ast})$ \\ {\tiny a Manin triple of Leibniz algebras}}
		\ar[d]^{\rm Lem. ~\ref{lem:man2}} \\
		\txt{$(A\otimes B, [-,-], \delta)$ \\ {\tiny a Lie bialgebra}}
		\ar@{<->}[r]^{\rm Pro. ~\ref{pro:man-Lie}\qquad\qquad} &
		\txt{$((A\otimes B)\oplus(A^{\ast}\otimes B^{\ast}),
			A\otimes B, A^{\ast}\otimes B^{\ast})$ \\
			{\tiny a Manin triples of Lie algebras}}}
\end{equation*}
\end{pro}

%%%%%%%%%%%%%%%%%%%%%%%%%%%%%%%%%%%%%%%%%%%%%%%%%%%%%%%%%%%%%%%%%%%%%%%%%%%
\subsection{Quasi-triangular Lie bialgebras from Leibniz bialgebras}\label{ssec:qtLi-Lib}
In this subsection, we study Lie bialgebras that arise from quasi-triangular Leibniz bialgebras, especially those that are triangular or factorizable.

Let $(B, \circ)$ be a Leibniz algebra. 
Define a linear map $F: B\rightarrow \End(B\otimes B)$ by
\begin{equation*}
	F(b):=\id\otimes\fr_{\circ}(b)-(\fl_{\circ}+\fr_{\circ})(b)\otimes\id, \;\; \forall b \in B.
\end{equation*}
An element $r \in B \otimes B$ is called {\bf Leib-invariant} if $F(b)(r)=0$ for all $b\in B$.
If there exists an element $r\in B\otimes B$ such that $(B, \circ, \vartheta)$ is a
Leibniz bialgebra, where $\vartheta: B\rightarrow B\otimes B$ is given by
\begin{equation}
	\vartheta(b) = F(b)(r),  \;\; \forall b\in B, \label{cobLeb}
\end{equation}
then $(B, \circ, \vartheta)$ is called a {\bf coboundary Leibniz bialgebra} associated with $r$. 
Let $(B, \circ)$ be a Leibniz algebra and $r=\sum_{i}x_{i}\otimes y_{i}\in B\otimes B$. 
The equation
\begin{equation*}
	\mathbf{L}_{r}=r_{12} \circ r_{13}-r_{12} \circ r_{23}-r_{23} \circ r_{12}+r_{23} \circ r_{13}=0
\end{equation*}
is called the (classical) {\bf Leibniz Yang-Baxter equation ($\LYBE$)} in $(B, \circ)$.
\begin{pro}[\cite{BLST, TS}]\label{pro:sLib-bia}
	Let $(B, \circ)$ be a Leibniz algebra, $r \in B \otimes B$ and $\vartheta: B \rightarrow B \otimes B$ defined by Eq.~\eqref{cobLeb}.
	If $r$ is a solution of the $\LYBE$ in $(B, \circ)$ and $r-\tau(r)$ is Leib-invariant, then $(B, \circ, \vartheta)$ is a Leibniz bialgebra, which is called a {\bf quasi-triangular Leibniz bialgebra} associated with $r$.
	In particular, if $r$ is a symmetric solution of the $\LYBE$ in $(B, \circ)$, then $(B, \circ, \vartheta)$ is a Leibniz bialgebra, which is called a {\bf triangular Leibniz bialgebra} associated with $r$.
\end{pro}

Let $B$ be a vector space. 
For any $r\in B\otimes B$, define a linear map $r^{\sharp}: B^{\ast}\rightarrow B$ by
\begin{equation*}
	\langle r^{\sharp}(\xi_{1}),\; \xi_{2}\rangle=\langle\xi_{1}\otimes\xi_{2},\; r\rangle, \;\; \forall \xi_{1}, \xi_{2}\in B^{\ast}.
\end{equation*}

\begin{defi}\label{def:fact-Leib}
Let $(B, \circ)$ be a Leibniz algebra, $r \in B \otimes B$ and $(B, \circ, \vartheta)$ a quasi-triangular Leibniz bialgebra associated with $r$.
If $\mathcal{I} = r^{\sharp}-\tau(r)^{\sharp}: B^{\ast}\rightarrow B$ is an isomorphism of vector spaces, 
then $(B, \circ, \vartheta)$ is called a {\bf factorizable Leibniz bialgebra}.
\end{defi}

Recall a {\bf $\bz$-graded Lie algebra} is a Lie algebra $(\g, [-,-]$ with a linear decomposition
$\g=\oplus_{i\in\bz}\g_{i}$ such that each $\g_{i}$ is finite-dimensional and $[\g_{i}, \g_{j}]\subseteq\g_{i+j}$ for all $i, j \in \bz$.
Let $(\g=\oplus_{i\in\bz}\g_{i}, [-,-])$ be a $\bz$-graded Lie algebra. 
Then $r=\sum_{i,j,\alpha}x_{i\alpha}\otimes y_{j\alpha}\in\g \hat{\otimes} \g$ is called {\bf completed Lie-invariant} if 
\begin{equation*}
	(\id \hat{\otimes} \ad(g) + \ad(g) \hat{\otimes} \id)(r) = 0, \;\; \forall g \in \g,
\end{equation*}
and called a {\bf completed solution of the CYBE} in $(\g, [-,-])$ if $r$ satisfies the classical Yang-Baxter equation ($\CYBE$):
\begin{equation*}
	\mathbf{C}_{r}:=[r_{12}, r_{13}]+[r_{12}, r_{23}]+[r_{13}, r_{23}]=0
\end{equation*}
as an element in $\g \hat{\otimes} \g\hat{\otimes} \g$.
When $\g = \g_{0}$, $r$ is simply called {\bf  Lie-invariant} and a {\bf solution of the CYBE} in $(\g, [-,-])$ respectively.

\begin{pro}[\cite{CP, LS}]\label{pro:comlie-bia}
	Let $(\g=\oplus_{i\in\bz}\g_{i}, [-,-])$ be a $\bz$-graded Lie algebra, $r\in\g\,\hat{\otimes}\,\g$ and $\delta: \g\rightarrow\g\,\hat{\otimes}\,\g$ be the linear map defined by
	\begin{equation}
		\delta(g) = (\id\,\hat{\otimes}\,\ad_{\g}(g)+\ad_{\g}(g)\,\hat{\otimes}\,\id)(r), \;\; \forall g \in \g.
		\label{coplie-cobo}
	\end{equation}
	If $r$ is a (completed) solution of the $\CYBE$ in $(\g, [-,-])$ and $r+\hat{\tau}(r)$ is Lie-invariant, then $(\g, [-,-], \delta)$ is a (completed) Lie bialgebra, which is called a {\bf quasi-triangular (completed) Lie bialgebra} associated with $r$.
	In particular, if $r$ is a skew-symmetric (completed) solution of the $\CYBE$ in
	$(\g, [-,-])$, then $(\g, [-,-], \delta)$ is a (completed) Lie bialgebra,
	which is called a {\bf triangular (completed) Lie bialgebra} associated with $r$.	
\end{pro}

\begin{defi}
	A (completed) quasi-triangular Lie bialgebra $(\g, [-,-], \delta)$ associated with $r$ is called a {\bf factorizable (completed) Lie bialgebra} if $\mathcal{I}=r^{\sharp}+\tau(r)^{\sharp}: \g^{\ast}\rightarrow\g$ is an isomorphism of vector spaces.
\end{defi}

We now consider the relation between the solutions of the $\LYBE$ in a Leibniz algebra and the solutions of the $\CYBE$ in the induced Lie algebra, leading to the following conclusion.
\begin{pro}\label{pro:LYBE-CYBE}
	Let $(B, \circ)$ be a Leibniz algebra, $(A, \cdot, \kappa)$ be a quadratic Zinbiel algebra and $(\g:=A\otimes B, [-,-])$ be the induced Lie algebra. 
	Let $\{e_{1}, e_{2},\cdots, e_{n}\}$ be a basis of $A$
	and $\{f_{1}, f_{2},\cdots, f_{n}\}$ be the dual basis of $\{e_{1}, e_{2},\cdots, e_{n}\}$ with respect to $\kappa$, i.e., $\kappa(f_{j}, e_{i})=\delta_{ij}$.
	If $r=\sum_{i}x_{i}\otimes y_{i}\in B\otimes B$ is a solution of the $\LYBE$ in $(B, \circ)$ and $r-\tau(r)$ is Leib-invariant,
	then
	\begin{equation}
		\tilde{r} := \sum_{i,j}(e_{j}\otimes x_{i})\otimes(f_{j}\otimes y_{i})
		\in(A\otimes B)\otimes(A\otimes B)  \label{r-max}
	\end{equation}
	is a solution of the $\CYBE$ in $(\g, [-,-])$ and $\tilde{r}+\tau(\tilde{r})$
	is Lie-invariant.
\end{pro}
\begin{proof}
	For all $s, u, v \in\{1, 2,\cdots, n\}$, we have
	\begin{align*}
		\widetilde{\kappa}(e_{s}\otimes e_{u}\otimes e_{v},\; \sum_{p,q} e_{p}\cdot e_{q} \otimes f_{p}\otimes f_{q}) &= \kappa(e_{s},\; e_{u} \cdot e_{v}),\\[-2mm]
		\widetilde{\kappa}(e_{s} \otimes e_{u} \otimes e_{v},\; \sum_{p,q}e_{p}\otimes f_{p} \cdot e_{q} \otimes f_{q} ) &=-\kappa(e_{s},\; e_{u}\cdot e_{v}).
	\end{align*}
	By the nondegeneracy of $\widetilde{\kappa}$, we obtain
	\begin{equation*}
		\sum_{p,q} e_{p} \cdot e_{q} \otimes f_{p} \otimes f_{q} = -\sum_{p,q} e_{p} \otimes f_{p}\cdot e_{q} \otimes f_{q}.
	\end{equation*}
	Similarly, we have
	\begin{align*}
		&\sum_{p,q}e_{q}\cdot e_{p}\otimes f_{p}\otimes f_{q}
		=\sum_{p,q}e_{p}\otimes e_{q}\otimes f_{p}\cdot f_{q}, \\
		&\sum_{p,q}e_{p}\otimes e_{q}\cdot f_{p} \otimes f_{q}=-\sum_{p,q}e_{p}\otimes e_{q}\otimes f_{q}\cdot f_{p}=\sum_{p,q}\big(e_{p}\cdot e_{q}\otimes f_{p}\otimes f_{q}
		 + e_{q}\cdot e_{p}\otimes f_{p}\otimes f_{q}\big)
	\end{align*}
	Therefore, we have
	\begin{align*}
	&\;[\tilde{r}_{12},\; \tilde{r}_{13}]+[\tilde{r}_{12},\; \tilde{r}_{23}]
	+[\tilde{r}_{13},\; \tilde{r}_{23}]\\
	=&\;\sum_{i,j}\sum_{p,q}\Big( (e_{p}\cdot e_{q} \otimes f_{p} \otimes f_{q})
	\bullet ( x_{i} \circ x_{j} \otimes y_{i} \otimes y_{j} )
	-( e_{q}\cdot e_{p} \otimes f_{p}\otimes f_{q})
	\bullet ( x_{j} \circ x_{i} \otimes y_{i}\otimes y_{j} )\\[-4mm]
	&\qquad\quad+(e_{p}\otimes f_{p}\cdot e_{q} \otimes f_{q})
	\bullet (x_{i}\otimes y_{i} \circ x_{j} \otimes y_{j} )
	-( e_{p}\otimes e_{q}\cdot f_{p} \otimes f_{q} )
	\bullet (x_{i}\otimes x_{j} \circ y_{i} \otimes y_{j})\\[-1mm]
	&\qquad\quad + ( e_{p}\otimes e_{q}\otimes f_{p}\cdot f_{q} )\bullet
	(x_{i} \otimes x_{j}\otimes y_{i} \circ y_{j} )
	-( e_{p}\otimes e_{q}\otimes f_{q}\cdot f_{p} )\bullet
	(x_{i}\otimes x_{j}\otimes y_{j} \circ y_{i}) \Big) \\
	=&\;\sum_{i,j}\sum_{p,q}\Big( \\
	&\quad ( e_{p}\cdot e_{q}\otimes f_{p}\otimes f_{q} ) \bullet
	( x_{i} \circ x_{j} \otimes y_{i} \otimes y_{j}
	-x_{i}\otimes y_{i} \circ x_{j} \otimes y_{j} - x_{i}\otimes x_{j} \circ y_{i} \otimes y_{j} + x_{i}\otimes x_{j} \otimes y_{j} \circ y_{i} )\\[-2mm]
	&+( e_{q} \cdot e_{p} \otimes f_{p}\otimes f_{q})\bullet
	(x_{i}\otimes x_{j}\otimes y_{i} \circ y_{j} 
	+x_{i}\otimes x_{j}\otimes y_{j} \circ y_{i} - x_{i}\otimes x_{j} \circ y_{i} \otimes y_{j} - x_{j} \circ x_{i} \otimes y_{i} \otimes y_{j})\Big).
	\end{align*}
	Since $r-\tau(r)$ is Leib-invariant, we get
	\begin{align*}
	\mathbf{L}_{r} &= \sum_{i,j}( x_{i} \circ x_{j} \otimes y_{i}\otimes y_{j}-x_{i}\otimes
	 y_{i} \circ x_{j} \otimes y_{j}-x_{j}\otimes x_{i} \circ y_{j} \otimes y_{i}
	+x_{j}\otimes x_{i}\otimes y_{i} \circ y_{j} )\\[-2mm]
	&=\sum_{i,j} ( x_{i} \circ x_{j} \otimes y_{i}\otimes y_{j}
	-x_{j}\otimes y_{j} \circ y_{i} \otimes x_{i}-x_{j}\otimes y_{i} \circ y_{j} \otimes x_{i} + x_{j}\otimes y_{i}\otimes x_{i} \circ y_{j} ).
	\end{align*}
	Hence, we have
	\begin{align*}
		&[\tilde{r}_{12},\; \tilde{r}_{13}]+[\tilde{r}_{12},\; \tilde{r}_{23}]
		+[\tilde{r}_{13},\; \tilde{r}_{23}] \\
		=&\Big(\sum_{p,q} e_{p}\cdot e_{q}\otimes f_{p}\otimes f_{q}\Big)\bullet\mathbf{L}_{r}-\Big(\sum_{p,q} e_{q} \cdot e_{p} \otimes f_{p} \otimes f_{q}\Big)\bullet(\id\otimes\tau)(\mathbf{L}_{r}) \\
		=&0.
	\end{align*}
	That is, $\tilde{r}$ is a solution of the $\CYBE$ in $(\g, [-,-])$. 
	Next, we show that $\tilde{r}+\tau(\tilde{r})$ is Lie-invariant.
	For all $b \in B$ and $p \in \{1, 2,\cdots, n\}$, we have
	\begin{align*}
	&\;\big(\id\otimes\ad(e_{p}\otimes b)+\ad(e_{p}\otimes b)
	\otimes\id\big)(\tilde{r}+\tau(\tilde{r}))\\
	=&\;\sum_{i,j}\Big( (e_{j}\otimes e_{p}\cdot f_{j} ) \bullet
	(x_{i}\otimes b \circ y_{i} )-(e_{j}\otimes f_{j}\cdot e_{p} )\bullet
	(x_{i}\otimes y_{i} \circ b )\\[-5mm]
	&\qquad+( e_{p}\cdot e_{j} \otimes f_{j})\bullet ( b \circ x_{i} \otimes y_{i})
	-( e_{j}\cdot e_{p} \otimes f_{j}) \bullet (x_{i} \circ b\otimes y_{i})\\[-1mm]
	&\qquad+(f_{j}\otimes e_{p}\cdot e_{j} )\bullet (y_{i}\otimes b \circ x_{i} )
	-(f_{j}\otimes e_{j}\cdot e_{p} )\bullet (y_{i}\otimes x_{i} \circ b)\\[-1mm]
	&\qquad+(e_{p}\cdot f_{j} \otimes e_{j})\bullet( b \circ y_{i}\otimes x_{i})
	-(f_{j}\cdot e_{p}\otimes e_{j})\bullet(y_{i} \circ b \otimes x_{i}) \Big).
	\end{align*}
	For all $s, t\in\{1, 2,\cdots, n\}$, noting that
	\begin{equation*}
		\widetilde{\kappa}(e_{s}\otimes e_{t},\; \sum_{j}e_{j}\otimes f_{j}\cdot e_{p} ) = \kappa(e_{s},\; e_{t}\cdot e_{p}) = -\widetilde{\kappa}(e_{s}\otimes e_{t},\; \sum_{j} e_{j} \cdot e_{p} \otimes f_{j}),
	\end{equation*}
	and $\widetilde{\kappa}$ is nondegenerate, we have
	\begin{equation*}
		\sum_{j}e_{j}\otimes f_{j}\cdot e_{p} 
		=-\sum_{j} e_{j}\cdot e_{p} \otimes f_{j}.
	\end{equation*}
	Similarly, we have 
	\begin{align*}
		&\sum_{j}e_{j}\otimes f_{j}\cdot e_{p} = \sum_{j} f_{j} \cdot e_{p} 
		\otimes e_{j}=-\sum_{j} f_{j} \otimes e_{j}\cdot e_{p} , \\
		&\sum_{j} e_{p}\cdot f_{j} 
		\otimes e_{j}=-\sum_{j} e_{p}\cdot e_{j} \otimes f_{j}, \\
		&\sum_{j}f_{j}\otimes e_{p}\cdot e_{j}=-\sum_{j}e_{j}\otimes e_{p}\cdot f_{j} =\sum_{j}e_{j}\otimes
		 f_{j}\cdot e_{p} +\sum_{j} e_{p}\cdot f_{j} \otimes e_{j}.
	\end{align*}
	Therefore, we have
	\begin{align*}
		&\;\big(\id\otimes\ad(e_{p}\otimes b)+\ad(e_{p}\otimes b)
		\otimes\id\big)(\tilde{r}+\tau(\tilde{r}))\\
		=&\;\sum_{i,j}\Big( ( e_{p}\cdot f_{j} \otimes e_{j})\bullet
		(y_{i}\otimes b \circ x_{i} + b \circ y_{i} \otimes x_{i}
		-x_{i}\otimes b \circ y_{i} - b \circ x_{i} \otimes y_{i})\\[-4mm]
		&\quad-(e_{j}\otimes f_{j} \cdot e_{p} )\bullet
		(x_{i}\otimes y_{i} \circ b +x_{i}\otimes b \circ y_{i} - x_{i} \circ b \otimes y_{i} -y_{i}\otimes b \circ x_{i} 
		-y_{i}\otimes x_{i} \circ b + y_{i} \circ b \otimes x_{i})\Big).
	\end{align*}
	Since $r-\tau(r)$ is Leib-invariant, we have
	\begin{align*}
	&\sum_{i} ( y_{i}\otimes b \circ x_{i} + b \circ y_{i} \otimes x_{i}
	- x_{i} \otimes b \circ y_{i} - b \circ x_{i} \otimes y_{i} ) = 0,\\[-2mm]
	&\sum_{i}(x_{i}\otimes y_{i} \circ b + x_{i}\otimes b \circ y_{i} - x_{i} \circ b \otimes y_{i}
	-y_{i}\otimes b \circ x_{i} -y_{i}\otimes x_{i} \circ b + y_{i} \circ b \otimes x_{i})=0.
	\end{align*}
	Thus, $\big(\id\otimes\ad(e_{p}\otimes b)+\ad(e_{p}\otimes b) \otimes\id\big)(\tilde{r}+\tau(\tilde{r})) = 0$ and $\tilde{r}+\tau(\tilde{r})$ is Lie-invariant.
\end{proof}

In the special case where $r$ is symmetric, we directly obtain the following corollary.
\begin{cor}\label{cor:sLYBE-sCYBE}
	Let $(B, \circ)$ be a Leibniz algebra, $(A, \cdot, \kappa)$ be a quadratic Zinbiel algebra and $(A \otimes B, [-,-])$ be the induced Lie algebra. 
	Let $\{e_{1}, e_{2},\cdots, e_{n}\}$ be a basis of $A$ and $\{f_{1}, f_{2},\cdots, f_{n}\}$ be the dual basis of $\{e_{1}, e_{2},\cdots, e_{n}\}$ with respect to $\kappa$.
	If $r = \sum_{i} x_{i} \otimes y_{i} \in B\otimes B$ is a symmetric solution of the $\LYBE$ in $(B, \circ)$, 
	then $\tilde{r}$ given by Eq.~\eqref{r-max} is a skew-symmetric solution of the $\CYBE$ in $(A \otimes B, [-,-])$.
\end{cor}
\begin{proof}
	It is sufficient to prove that $\tilde{r}$ is skew-symmetric.
	For all $s, t \in \{1, \cdots, n\}$, we have
	\begin{equation*}
		\widetilde{\kappa}(e_{s}\otimes e_{t},\; \sum_{j}e_{j}\otimes f_{j}) = -\kappa(e_{s}, e_{t})=\kappa(e_{t}, e_{s}) = -\widetilde{\kappa}(e_{s}\otimes e_{t},\; \sum_{j}f_{j}\otimes e_{j}).
	\end{equation*}
	The nondegeneracy of $\widetilde{\kappa}$ yields that $\sum_{j}e_{j}\otimes f_{j} =-\sum_{j}f_{j}\otimes e_{j}$. 
	Noting that $r$ is symmetric, we show that $\tilde{r}$ is skew-symmetric.
	The proof is completed.
\end{proof}

\begin{thm}\label{thm:indu-sLiebia-lei}
	With the same assumptions in Proposition~\ref{pro:LYBE-CYBE}.
	Let $\vartheta: B \rightarrow B \otimes B$ be the linear map defined by Eq.~\eqref{cobLeb}, $\Delta: A \rightarrow A \otimes A$ be the linear map defined by Eq.~\eqref{eq:q2zc} and $\delta: \g \rightarrow \g \otimes \g$ be the linear map defined by Eq.~\eqref{colie}.
	Then $(B, \circ, \vartheta)$ is a Leibniz bialgebra by Proposition~\ref{pro:sLib-bia} and hence $(\g, [-,-], \delta)$ is a Lie bialgebra by Theorem~\ref{thm:liebia-LZ}.
	It coincides with the quasi-triangular Lie bialgebra with $\delta$ defined by Eq.~\eqref{coplie-cobo} through $\tilde{r}$ by Proposition~\ref{pro:comlie-bia}, where $\tilde{r}$ is defined by Eq.~\eqref{r-max}. 
	Hence, we have the following commutative diagram.
	\begin{equation*}
	\xymatrix@C=3cm@R=0.6cm{
		\txt{$r$ \\ {\tiny a solution of the $\LYBE$ in $(B, \circ)$}\\
			{\tiny such that $r-\tau(r)$ is Leib-invariant}}
		\ar[d]_{{\rm Pro.}~\ref{pro:LYBE-CYBE}}\ar[r]^{{\rm Pro.}~\ref{pro:sLib-bia}} &
		\txt{$(B, \circ, \vartheta)$ \\ {\tiny a quasi-triangular Leibniz bialgebra}}
		\ar[d]^{{\rm Thm.}~\ref{thm:liebia-LZ}} \\
		\txt{$\tilde{r}$ \\ {\tiny a solution of the $\CYBE$ in $(A\otimes B, [-,-])$}\\
			{\tiny such that $\tilde{r}+\tau(\tilde{r})$ is Lie-invariant}}
		\ar[r]^{{\rm Pro.}~\ref{pro:comlie-bia}} &
		\txt{$(A\otimes B, [-,-], \delta)$ \\
			{\tiny the induced quasi-triangular Lie bialgebra}}}
	\end{equation*}
	In particular, if $(B, \circ, \vartheta)$ is triangular (resp. factorizable), then $(\g, [-,-], \delta)$ is triangular (resp. factorizable).
\end{thm}
\begin{proof}
	By Proposition~\ref{pro:sLib-bia}, $(B, \circ, \vartheta)$ is a Leibniz bialgebra.
	Then by Theorem~\ref{thm:liebia-LZ}, there is a Lie bialgebra structure $(\g, [-,-], \delta)$ on $\g$ where $\delta$ is defined by Eq.~\eqref{colie}.
	For all basis elements $e_{s}, e_{t}\in A$, we have
	\begin{equation*}
		\widetilde{\kappa}(e_{s}\otimes e_{t},\; \Delta(a)) = \kappa(e_{s},\; a\cdot e_{t}) = -\widetilde{\kappa}(e_{s}\otimes e_{t},\; \sum_{j} a\cdot e_{j} \otimes f_{j})
	\end{equation*}
	Since $\widetilde{\kappa}$ is nondegenerate, we obtain 
	\begin{equation*}
		\Delta(a) = -\sum_{j}a \cdot e_{j} \otimes f_{j}.
	\end{equation*}
	Similarly, we have 
	\begin{align*}
		&\sum_{j} e_{j} \otimes f_{j}\cdot a =-\sum_{j} e_{j}\cdot a \otimes f_{j}, \\
		&\sum_{(a)} \tau \Delta(a) = \sum_{j} e_{j} \otimes a \cdot f_{j} = \sum_{j} a \cdot e_{j} \otimes f_{j} +
		\sum_{j} e_{j} \cdot a \otimes f_{j}.
	\end{align*}
	Therefore, 
	\begin{align*}
		&\delta(a\otimes b) - (\id\otimes\ad_{\g}(a\otimes b)
		+\ad_{\g}(a\otimes b)\otimes\id)(\sum_{i, j}(e_{j}\otimes x_{i})\otimes(f_{j}\otimes y_{i}))\\
		&=\sum_{i}\Big(
		\Delta(a) \bullet (x_{i}\otimes y_{i} \circ b )
		-\Delta(a) \bullet ( b \circ x_{i}\otimes y_{i}) - \Delta(a) \bullet (x_{i} \circ b \otimes y_{i} ) \\[-5mm]
		&\qquad\qquad - \tau \Delta(a) \bullet ( y_{i} \circ b \otimes x_{i} ) + \tau \Delta(a) \bullet (y_{i}\otimes b \circ x_{i} )
		+\tau \Delta(a)\bullet(y_{i}\otimes x_{i} \circ b)\Big) \\
		&\quad- \sum_{i,j}\Big((e_{j}\otimes a \cdot f_{j})\bullet(x_{i}\otimes b \circ y_{i} )
		-(e_{j}\otimes f_{j}\cdot a )\bullet(x_{i}\otimes y_{i} \circ b )\\[-5mm]
		&\qquad\quad+(a \cdot e_{j}\otimes f_{j})\bullet(b \circ x_{i} \otimes y_{i})
		-( e_{j}\cdot a \otimes f_{j})\bullet( x_{i} \circ b\otimes y_{i})\Big) \\
		&=\; \tau\Delta(a)\bullet
		\Big(\sum_{i} x_{i} \circ b \otimes y_{i}-x_{i}\otimes y_{i} \circ b +y_{i}\otimes
		 b \circ x_{i}+y_{i}\otimes x_{i} \circ b
		- y_{i} \circ b \otimes x_{i}-x_{i}\otimes b \circ y_{i} \Big).
	\end{align*}
 	Note that $r-\tau(r)$ is Leib-invariant, i.e.,
 	\begin{equation*}
 		\sum_{i} b \circ x_{i} \otimes y_{i}+ x_{i} \circ b 
 		\otimes y_{i}-x_{i}\otimes y_{i} \circ b - b \circ y_{i} \otimes x_{i}- y_{i} \circ b
 		\otimes x_{i}+y_{i}\otimes x_{i} \circ b =0.
 	\end{equation*}
 	Thus, $\delta(a\otimes b) = (\id\otimes\ad_{\g}(a\otimes b)
 	+\ad_{\g}(a\otimes b)\otimes\id)(\tilde{r})$.
 	In particular, if $(B, \circ, \vartheta)$ is triangular, then we show that $(A\otimes B, [-,-], \delta)$ is triangular by Corollary
 	\ref{cor:sLYBE-sCYBE}. 
 	
 	Now suppose that $(B, \circ, \vartheta)$ is factorizable. 
 	Immediately, we have $\mathcal{I}=r^{\sharp} -\tau(r)^{\sharp}$ is an isomorphism of vector spaces.
	To prove $(\g, [-,-], \delta)$ is factorizable, it is sufficient to show that the map $\tilde{\mathcal{I}}=\tilde{r}^{\sharp} +\tau(\tilde{r})^{\sharp}: \g^{\ast}\rightarrow \g$ is an isomorphism of vector spaces. 
	Denote $r_{0}:=\sum_{j}e_{j}\otimes
	f_{j}\in A\otimes A$. 
	It is easy to see that $r_{0}^{\sharp}$ is
	a linear isomorphism, $\tilde{r}^{\sharp}=r_{0}^{\sharp}\otimes r^{\sharp}$ and $\tau(\tilde{r})^{\sharp}=-r_{0}^{\sharp}\otimes\tau(r)^{\sharp}$. 
	Therefore, $\tilde{\mathcal{I}}
	=\tilde{r}^{\sharp}+\tau(\tilde{r})^{\sharp}=r_{0}^{\sharp}\otimes(r^{\sharp}
	-\tau(r)^{\sharp})=r_{0}^{\sharp}\otimes\mathcal{I}$ is an isomorphism of vector
	spaces. 
\end{proof}

In \cite{Kup}, Kupershmidt found that the $\CYBE$ in tensor form on Lie algebras can be converted into an $\mathcal{O}$-operator associated to the coadjoint representation.
Let $(\g, [-,-])$ be a Lie algebra. 
Recall that a representation of $(\g, [-,-])$ is a pair $(V, \rho)$, where $V$ is a vector space and $\rho: \g\rightarrow \End(V)$ is a linear map such that $\rho([g_{1}, g_{2}])=\rho(g_{1})\rho(g_{2})-\rho(g_{2})\rho(g_{1})$ for all $g_{1}, g_{2}\in\g$. 
A linear map $\mathcal{T}: V\rightarrow\g$ is called an
{\bf $\mathcal{O}$-operator of $(\g, [-,-])$ associated to $(V, \rho)$} if
\begin{equation*}
	[\mathcal{T}(v_{1}), \mathcal{T}(v_{2})]
	=\mathcal{T}\big(\rho(\mathcal{T}(v_{1}))(v_{2})-\rho(\mathcal{T}(v_{2}))(v_{1})\big), \;\; \forall v_{1}, v_{2}\in V.
\end{equation*}

\begin{pro}[\cite{CP,Kup}]\label{pro:o-lie}
Let $(\g, [-,-])$ be a Lie algebra and $r\in\g\otimes\g$ be skew-symmetric.
Then $r$ is a solution of the $\CYBE$ in $(\g, [-,-])$ if and only if $r^{\sharp}$
is an $\mathcal{O}$-operator of $(\g, [-,-])$ associated to the coadjoint
representation $(\g^{\ast}, -\ad^{\ast})$.
\end{pro}

Let $(B, \circ)$ be a Leibniz algebra and $(V, \kl, \kr)$ be a representation of $(B, \circ)$. 
Recall that a linear map $\mathcal{T}: V\rightarrow B$ is called an {\bf $\mathcal{O}$-operator of $(B, \circ)$ associated to $(V, \kl, \kr)$} if 
\begin{equation*}
	\mathcal{T}(v_{1}) \circ \mathcal{T}(v_{2})=
	\mathcal{T}\big(\kl(\mathcal{T}(v_{1}))(v_{2})+\kr(\mathcal{T}(v_{2}))(v_{1})\big), \;\; \forall v_{1}, v_{2}\in V.
\end{equation*}

\begin{pro}[\cite{BLST,TS}]\label{pro:o-leib}
Let $(B, \circ)$ be a Leibniz algebra and $r\in B\otimes B$ be symmetric.
Then $r$ is a solution of the $\LYBE$ in $(B, \circ)$ if and only if $r^{\sharp}: B^{\ast}\rightarrow B$ is an $\mathcal{O}$-operator of $(B, \circ)$ associated to the coregular representation $(B^{\ast}, -\fl_{\circ}^{\ast},
\fl_{\circ}^{\ast}+\fr_{\circ}^{\ast})$.
\end{pro}

With the same assumptions in Theorem~\ref{thm:indu-sLiebia-lei}.
By the proof of Theorem~\ref{thm:indu-sLiebia-lei},
we have the following commutative diagram:
\begin{equation*}
	\xymatrix@C=3cm@R=0.5cm{
		\txt{$r$ \\ {\tiny a symmetric solution} \\ {\tiny of the $\LYBE$ in $(B, \circ)$}}
		\ar[d]_-{{\rm Pro.}~\ref{pro:LYBE-CYBE}}\ar[r]^-{{\rm Pro.}~\ref{pro:o-leib}} &
		\txt{$r^{\sharp}$\\ {\tiny an $\mathcal{O}$-operator of $(B, \circ)$} \\
			{\tiny associated to $(B^{\ast}, -\fl_{\circ}^{\ast}, \fr_{\circ}^{\ast}+\fl_{\circ}^{\ast})$}}
		\ar[d]^-{\mbox{$r_{0}^{\sharp}\otimes-$}} \\
		\txt{$\tilde{r}$ \\ {\tiny a skew-symmetric solution} \\ {\tiny of the $\CYBE$ in
				$(A\otimes B, [-,-])$}} \ar[r]^-{{\rm Pro.}~\ref{pro:o-lie}}
		& \txt{$\tilde{r}^{\sharp}=r_{0}^{\sharp}\otimes r^{\sharp}$ \\
			{\tiny an $\mathcal{O}$-operator of $(A\otimes B, [-,-])$ } \\
			{\tiny associated to $((A\otimes B)^{\ast}, -\ad^{\ast})$}}}
\end{equation*}

At the end of this section, we give an example illustrating the construction of a triangular Lie bialgebra arising from a triangular Leibniz bialgebra.

\begin{ex}\label{ex:tri-tensor}
Let $(A, \cdot, \kappa)$ be the quadratic Zinbiel algebra given in Example~\ref{ex:qu-zib}. 
Clearly, $\{e_{3}$, $e_{4}$, $-e_{1}$, $-e_{2}\}$ is the dual basis of $\{e_{1}$, $e_{2}$, $e_{3}$, $e_{4}\}$ with respect to $\kappa$.
Let $(B, \circ)$ be a 2-dimensional Leibniz algebra with a basis $\{x_{1}, x_{2}\}$ whose non-zero product $\circ$ is given by $x_{2} \circ x_{2}=x_{1}$. 
Moreover $(B, \circ, \vartheta)$ is a triangular Leibniz bialgebra associated with $r=x_{1}\otimes x_{2}+x_{2}\otimes x_{1}$, explicitly $\vartheta$ is defined by $\vartheta(x_{1})=0$
and $\vartheta(x_{2})=-x_{1}\otimes x_{1}$. 
By Theorem \ref{thm:liebia-LZ}, we get a Lie bialgebra $(A\otimes B, [-,-], \delta)$, where the non-zero product $[-,-]$ and
coproduct $\delta$ are given by
\begin{align*}
&\qquad -[e_{4}\otimes x_{2} x_{2},\; e_{1}\otimes] = [e_{1}\otimes x_{2},\; e_{4}\otimes x_{2}]
=3e_{3}\otimes x_{1}-3e_{2}\otimes x_{1},\\
&\delta(e_{1}\otimes x_{2})=-3(e_{2}\otimes x_{1})\otimes(e_{3}\otimes x_{1})
+3(e_{3}\otimes x_{1})\otimes(e_{2}\otimes x_{1})=\delta(e_{4}\otimes x_{2}),
\end{align*}
which is a triangular Lie bialgebra associated with 
\begin{align*}
\tilde{r}:=&(e_{1}\otimes x_{1})\otimes(e_{3}\otimes x_{2})
+(e_{1}\otimes x_{2})\otimes(e_{3}\otimes x_{1})
+(e_{2}\otimes x_{1})\otimes(e_{4}\otimes x_{2})
+(e_{2}\otimes x_{2})\otimes(e_{4}\otimes x_{1})\\[-1mm]
&\quad -(e_{3}\otimes x_{1})\otimes(e_{1}\otimes x_{2})
-(e_{3}\otimes x_{2})\otimes(e_{1}\otimes x_{1})
-(e_{4}\otimes x_{1})\otimes(e_{2}\otimes x_{2})
-(e_{4}\otimes x_{2})\otimes(e_{2}\otimes x_{1}).
\end{align*}
Moreover, $r$ induces an $\mathcal{O}$-operator $r^{\sharp}: B^{\ast}\rightarrow B$,
which is given by $r^{\sharp}(\eta_{1})=x_{2}$ and $r^{\sharp}(\eta_{2})=x_{1}$, where
$\eta_{1}, \eta_{2}\in B^{\ast}$ is the dual basis of $x_{1}, x_{2}$. 
Denote $r_{0}:=e_{1}\otimes e_{3}+e_{2}\otimes e_{4}-e_{3}\otimes e_{1}-e_{4}\otimes e_{2}$.
The linear map $r_{0}^{\sharp}: A^{\ast}\rightarrow A$ is given by $r_{0}^{\sharp}(\xi_{1})=e_{3}$,
$r_{0}^{\sharp}(\xi_{2})=e_{2}$, $r_{0}^{\sharp}(\xi_{3})=-e_{1}$ and $r_{0}^{\sharp}(\xi_{4})=-e_{2}$. 
Then, $r_{0}^{\sharp}\otimes r^{\sharp}$ is the $\mathcal{O}$-operator $\tilde{r}^{\sharp}: (A\otimes B)^{\ast} \rightarrow(A\otimes B)$ induced by $\tilde{r}$ given by
\begin{align*}
	&\tilde{r}^{\sharp}(\xi_{1}\otimes\eta_{1})=e_{3}\otimes x_{2},\;
	&&\tilde{r}^{\sharp}(\xi_{1}\otimes\eta_{2})=e_{3}\otimes x_{1},\;
	&&\tilde{r}^{\sharp}(\xi_{2}\otimes\eta_{1})=e_{4}\otimes x_{2},\;
	&&\tilde{r}^{\sharp}(\xi_{2}\otimes\eta_{2})=e_{4}\otimes x_{1},\\
	&\tilde{r}^{\sharp}(\xi_{3}\otimes\eta_{1})=-e_{1}\otimes x_{2},
	&&\tilde{r}^{\sharp}(\xi_{3}\otimes\eta_{2})=-e_{1}\otimes x_{1},
	&&\tilde{r}^{\sharp}(\xi_{4}\otimes\eta_{1})=-e_{2}\otimes x_{2},
	&&\tilde{r}^{\sharp}(\xi_{4}\otimes\eta_{2})=-e_{2}\otimes x_{1},
\end{align*}
where $\xi_{1}, \xi_{2}, \xi_{3}, \xi_{4}\in A^{\ast}$ is the dual basis of $e_{1}, e_{2},
e_{3}, e_{4}$. 
\end{ex}

\smallskip
%%%%%%%%%%%%%%%%%%%%%%%%%%%%%%%%%%%%%%%%%%%%%%%%%%%%%%%%%%%%%%%%%%%%%%%%%%%%%%%%%
%    section  4  Infinite-dimensional Lie bialgebras
%%%%%%%%%%%%%%%%%%%%%%%%%%%%%%%%%%%%%%%%%%%%%%%%%%%%%%%%%%%%%%%%%%%%%%%%%%%%%%%%%%%%%%

\section{Infinite-dimensional Lie bialgebras from Zinbiel bialgebras}\label{sec:inf-lie}
In this section, we construct an infinite-dimensional Lie bialgebra from a finite-dimensional Zinbiel bialgebra and an infinite-dimensional quadratic Leibniz algebra.
We show that a symmetric solution of the Zinbiel Yang-Baxter equation in a Zinbiel algebra gives rise to a skew-symmetric completed solution of the classical Yang-Baxter equation in the induced infinite-dimensional Lie algebra. 
Moreover, we present a construction of quasi-Frobenius $\bz$-graded Lie algebras from quasi-Frobenius Zinbiel algebras.

\begin{defi}[\cite{Wan}]\label{def:Zinbiel-bi}
	A {\bf Zinbiel bialgebra} is a triple $(A, \cdot, \Delta)$, where $(A, \cdot)$ is a Zinbiel algebra and $(A, \Delta)$ is a Zinbiel coalgebra satisfying
	\begin{align}
	&\Delta(a_{1}\cdot a_{2})+\Delta(a_{2}\cdot a_{1})
	=(\id\otimes(\fl_{\cdot}+\fr_{\cdot})(a_{2}))\Delta(a_{1})
	+(\fl_{\cdot}(a_{1})\otimes\id)\Delta(a_{2}),                 \label{zbialg1}\\
	&\Delta(a_{1}\cdot a_{2})+\tau\Delta(a_{1}\cdot a_{2})
	=(\id\otimes\fr_{\cdot}(a_{2}))\Delta(a_{1})
	+(\fl_{\cdot}(a_{1})\otimes\id)(\Delta(a_{2})+\tau\Delta(a_{2})),  \;\; \forall a_{1}, a_{2} \in A.\label{zbialg2}
	\end{align}
\end{defi}

Zinbiel bialgebras admit equivalent characterizations in terms of Manin triples of Zinbiel algebras and matched pairs of Zinbiel algebras.

\begin{thm}[\cite{Wan}]\label{thm:equat}
	Let $(A, \cdot)$ be a Zinbiel algebra and $\Delta: A \rightarrow A\otimes A$ be a linear map such that $\circ = \Delta^{\ast}: A^{\ast}\otimes A^{\ast} \rightarrow A^{\ast}$ defines a Zinbiel algebra structure on $A^{\ast}$.
	Then the following conditions are equivalent:
	\begin{enumerate}[(i)]
		\item
		$(A, \cdot, \Delta)$ is a Zinbiel bialgebra;
		
		\item
		$(A, A^{\ast}, \fl_{\cdot}^{\ast}+\fr_{\cdot}^{\ast}, -\fr_{\cdot}^{\ast}, \fl_{\circ}^{\ast}+\fr_{\circ}^{\ast}, -\fr_{\circ}^{\ast})$ is a matched pair of Zinbiel algebras;
		
		\item
		$((A\oplus A^{\ast}, \diamond, \mathcal{B}_{d}),\; A,\; A^{\ast})$ is a (standard) Manin triple of Zinbiel algebras.
	\end{enumerate}
\end{thm}

Recall that a Zinbiel bialgebra $(A, \cdot, \Delta)$ is called {\bf coboundary} if there exists $r \in A\otimes A$ such that the coproduct $\Delta$ is given by
\begin{equation}
	\Delta(a)=\big(\id\otimes(\fl_{\cdot}+\fr_{\cdot})(a) - \fl_{\cdot}(a)\otimes\id\big)(r),  \;\; \forall a \in A. \label{cobou}
\end{equation}

\begin{defi}[\cite{Wan}]\label{def:ZYBE}
	Let $(A, \cdot)$ be a Zinbiel algebra and $r\in A\otimes A$. 
	The equation
	\begin{equation*}
		\mathbf{Z}_{r}:=r_{13}\cdot r_{21}+r_{21}\cdot r_{13}+r_{12}\cdot r_{23}+r_{23}\cdot r_{12}
		-r_{13}\cdot r_{23}-r_{23}\cdot r_{13}-r_{13}\cdot r_{12}-r_{23}\cdot r_{21}=0
	\end{equation*}
	is called the (classical) {\bf Zinbiel Yang-Baxter equation ($\ZYBE$)} in $(A, \cdot)$.
\end{defi}

\begin{defi}\label{def:Z-inv}
	Let $(A, \cdot)$ be a Zinbiel algebra.
	An element $r \in A\otimes A$ is called {\bf Zinb-invariant} if
	\begin{equation*}
		\big(\id\otimes(\fl_{\cdot}+\fr_{\cdot})(a)-\fl_{\cdot}(a)\otimes\id\big)(r)=0, \;\; \forall a \in A.
	\end{equation*}
\end{defi}

\begin{pro}[\cite{Wan}]\label{pro:qtr-zib}
	Let $(A, \cdot)$ be a Zinbiel algebra, $r\in A\otimes A$ and $\Delta: A \rightarrow A\otimes A$ be a linear map defined by Eq.~\eqref{cobou}.
	If $r$ is a solution of the $\ZYBE$ in $(A, \cdot)$ and $r-\tau(r)$ is Zinb-invariant, then $(A, \cdot, \Delta)$ is a Zinbiel bialgebra, which is called a {\bf quasi-triangular} Zinbiel bialgebra associated with $r$.
	In particular, if $r$ is a symmetric solution of the $\ZYBE$ in $(A, \cdot)$, then $(A, \cdot, \Delta)$ is a Zinbiel bialgebra, which is called a {\bf triangular} Zinbiel bialgebra associated with $r$.
\end{pro}

\begin{ex}\label{ex:ZYBE}
	Let $(A, \cdot)$ be a 2-dimensional Zinbiel algebra with a basis $\{e_{1}, e_{2}\}$ whose non-zero products $\cdot$ is given by $e_{1}\cdot e_{1}=e_{2}$. 
	Clearly, $r = e_{1} \otimes e_{2} + e_{2} \otimes e_{1}$ is a symmetric solution of the $\ZYBE$ in $(A, \cdot)$.
	Therefore, $(A, \cdot, \Delta)$ is a triangular Zinbiel bialgebra, where $\Delta$ is given by Eq.~\eqref{cobou}, explicitly $\Delta$ is defined by $\Delta(e_{1})=e_{2}\otimes e_{2}$ and $\Delta(e_{2})=0$.
\end{ex}

%%%%%%%%%%%%%%%%%%%%%%%%%%%%%%%%%%%%%%%%%%%%%%%%%%%%%%%%%%%%%%%%%%%%%%%%%%%
\subsection{Lie bialgebras from Zinbiel bialgebras}\label{ssec:Li-Zin}

The pairing of a Zinbiel algebra and a $\bz$-graded Leibniz algebra gives a $\bz$-graded Lie algebra, as stated below. 
This construction is referred to as the Zinbiel algebra affinization.

\begin{pro}\label{pro:L-Z-inflie}
	Let $(A, \cdot)$ be a Zinbiel algebra and $(B, \circ)$ be a $\bz$-graded Leibniz algebra. 
	Define a binary operation on $A \otimes B$ by
	\begin{equation*}
		[a\otimes x,\; b\otimes y]= a \cdot b \otimes x \circ y - b \cdot a \otimes y \circ x, \;\; \forall a, b \in A, \; x, y \in B.
	\end{equation*}
	Then $(A\otimes B,\; [-, -])$ is a $\bz$-graded Lie algebra.
	In particular, if $(B, \circ)$ is the $\bz$-graded Leibniz algebra $(\widehat{V}_{4}, \circ)$ given in Example~\ref{ex:gr-Leib}, then $(A \otimes B,\; [-, -])$ is a $\bz$-graded Lie algebra if and only if $(A, \cdot)$ is a Zinbiel algebra.
\end{pro}
\begin{proof}
	By Proposition~\ref{pro:L-Z-lie}, $(A \otimes B,\; [-, -])$ is a Lie algebra.
	Clearly, $(A\otimes B,\; [-, -])$ is a $\bz$-graded Lie algebra for $(B, \circ)$ is $\bz$-graded. 
	Suppose $(B, \circ)$ is the $\bz$-graded Leibniz algebra $(\widehat{V}_{4}, \circ)$ in Example~\ref{ex:gr-Leib} and $(A\otimes B,\; [-, -])$ is a $\bz$-graded Lie algebra.
	Set $b_{1}=v_{1}\kt^{i}$, $b_{2}=v_{2}\kt^{i}$ and $b_{3}=v_{3}\kt^{i}$.
	By comparing the coefficients of $v_{4}\kt^{3i}$ in the equation
	\begin{equation*}
		[[a_{1}\otimes b_{1},\; a_{2}\otimes b_{2}],\; a_{3}\otimes b_{3}]
		+[[a_{2}\otimes b_{2},\; a_{3}\otimes b_{3}],\; a_{1}\otimes b_{1}]
		+[[a_{3}\otimes b_{3},\; a_{1}\otimes b_{1}],\; a_{2}\otimes b_{2}]=0,
	\end{equation*}
	we get 
	\begin{equation*}
		a_{1} \cdot(a_{2} \cdot a_{3}) - (a_{1} \cdot a_{2}) \cdot a_{3} - (a_{2} \cdot a_{1}) \cdot a_{3}=0.
	\end{equation*}
	That is, $(A, \cdot)$ is a Zinbiel algebra.
\end{proof}

Now we give the dual version of Proposition~\ref{pro:L-Z-inflie}.

\begin{ex}\label{ex:zin-coalg}
	Let $V$ be a 4-dimensional vector space with a basis $\{v_{1}, v_{2}, v_{3}, v_{4}\}$.
	Then $(V[t, t^{-1}]  = \oplus_{i\in\bz}V_{i}, \vartheta)$ is a completed Leibniz coalgebra, where $V_{i}=\Bbbk\{v_{1}\kt^{i}, v_{2}\kt^{i}, v_{3}\kt^{i}, v_{4}\kt^{i}\}$ and $\vartheta: V[\kt, \kt^{-1}] \rightarrow V[\kt, \kt^{-1}] \,\hat{\otimes}\, V[\kt, \kt^{-1}]$ is defined by 
	\begin{align*}
		&\vartheta(v_{1}\kt^{i})=\sum_{j}v_{4}\kt^{j}\otimes v_{1}\kt^{i-j}, \;
		&&\vartheta(v_{2}\kt^{i})=-\sum_{j}v_{3}\kt^{j}\otimes v_{1}\kt^{i-j}, \\
		&\vartheta(v_{3}\kt^{i})=\sum_{j}\big(v_{3}\kt^{j}\otimes v_{4}\kt^{i-j}
		-v_{4}\kt^{j}\otimes v_{3}\kt^{i-j}\big), \; 
		&&\vartheta(v_{4}\kt^{i})=0.
	\end{align*}
\end{ex}

\begin{pro}\label{pro:lie-co}
	Let $(A, \Delta)$ be a Zinbiel coalgebra and $(B, \vartheta)$ be a completed Leibniz coalgebra. 
	Define a linear map $\delta: A\otimes B\rightarrow(A\otimes B) \,\hat{\otimes}\,(A\otimes B)$ by Eq.~\eqref{colie}.
	Then $(A\otimes B, \delta)$ is a completed Lie coalgebra.
	In particular, if $(B, \vartheta)$ is the completed Leibniz coalgebra given in Example \ref{ex:zin-coalg}, then
	$(A \otimes B,\; \delta)$ is a completed Lie algebra if and only if $(A, \Delta)$ is a Zinbiel coalgebra.
\end{pro}
\begin{proof}
	By Proposition~\ref{pro:L-Z-colie}, $(A\otimes B, \delta)$ is a completed Lie coalgebra. 
	Conversely, suppose that $(B,\; \vartheta)$ is the completed Leibniz coalgebra $(V[\kt, \kt^{-1}],\; \vartheta)$ and $(A \otimes B,\; \delta)$ is a completed Lie algebra.
	Setting $b=v_{2}\kt^{i}$ and comparing the coefficients of $\sum_{j,k}v_{3}\kt^{j}\otimes v_{3}\kt^{k}\otimes v_{3}\kt^{i-j-k}$ in the equation
	\begin{equation*}
		\big((\id\,\hat{\otimes}\,\delta)\delta-(\hat{\tau}\,\hat{\otimes}\,\id) (\id\,\hat{\otimes}\,\delta)\delta-(\delta\,\hat{\otimes}\,\id)\delta\big)(a\otimes b)=0,
	\end{equation*}
	we obtain
	\begin{equation*}
		(\id\otimes\Delta)(\Delta(a))=(\Delta\otimes\id)(\Delta(a))
		+(\tau\otimes\id)((\Delta\otimes\id)(\Delta(a))), \;\; \forall a \in A.
	\end{equation*}
	Therefore, $(A, \Delta)$ is a Zinbiel coalgebra.
\end{proof}

Now we could extend the Zinbiel algebra affinization and the Zinbiel coalgebra
affinization to Zinbiel bialgebras.

\begin{defi}\label{def:quad}
	A bilinear form $\omega$ on a $\bz$-graded Leibniz algebra $(B=\oplus_{i\in\bz}B_{i},\; \circ)$ is called {\bf graded}, if there exists some $m\in\bz$ such that $\omega(B_{i}, B_{j})=0$ when $i+j+m \neq 0$.
	A {\bf quadratic $\bz$-graded Leibniz algebra}, denoted by $(B=\oplus_{i\in\bz}B_{i}, \circ, \omega)$, is a $\bz$-graded Leibniz algebra together with a skew-symmetric invariant nondegenerate graded bilinear form.
	In particular, if $B=B_{0}$, it is simply called a {\bf quadratic Leibniz algebra}.
\end{defi}

\begin{rmk}
	Let $(B=\oplus_{i\in\bz}B_{i}, \circ, \omega)$ be a quadratic $\bz$-graded Leibniz algebra
	and $\{e_{\lambda}\}_{\lambda\in\Lambda}$ be a basis of $B$ consisting of homogeneous
	elements. 
	Then we can always find its graded dual basis $\{f_\lambda\}$ with respect to $\omega$ consisting of homogeneous elements, that is, $\omega(f_{\lambda}, e_{\eta})
	=\delta_{\lambda, \eta}$ for all $\lambda, \eta \in \Lambda$.
\end{rmk}

Let $\omega$ be a nondegenerate graded bilinear form on a $\mathbb{Z}$-graded vector space $V = \oplus_{i \in \mathbb{Z}} V_i$. 
For all $l \geq 2$, the bilinear form $\widetilde{\omega}_l: \overbrace{(V \hat{\otimes} \cdots \hat{\otimes} V)}^{l-\text{fold}} \otimes \overbrace{(V \otimes \cdots \otimes V)}^{l-\text{fold}} \to \Bbbk$ defined by
\begin{equation*}
	\widetilde{\omega}_l(\sum_{i_1, \cdots, i_l, \alpha} v_{1i_1\alpha} \otimes \cdots \otimes v_{li_l\alpha}, \sum_{j_1, \cdots, j_l}u_{j_1} \otimes \cdots \otimes u_{j_l}) := \sum_{i_1,\cdots, i_k, \alpha} \sum_{j_1, \cdots, j_l} \prod_{m=1}^l \omega(v_{mi_m \alpha}, u_{j_m}),
\end{equation*}
is \textbf{left nondegenerate} in the sense that
$\widetilde{\omega}_l(\sum_{i_1, \cdots, i_l, \alpha}
v_{1i_1\alpha} \otimes \cdots \otimes v_{li_l\alpha}, u_1 \otimes
\cdots \otimes u_l) = 0$ for all homogeneous elements $u_1, \cdots, u_l \in V$ implies $\sum_{i_1, \cdots, i_l, \alpha} v_{1i_1\alpha} \otimes \cdots \otimes v_{li_l\alpha} = 0$. 
For brevity, we will suppress the index $l$ without ambiguity.

\begin{lem}\label{lem:dual}
	Let $(B=\oplus_{i\in\bz}B_{i},\; \circ,\; \omega)$ be a quadratic $\bz$-graded Leibniz algebra and $\vartheta: B\rightarrow B\,\hat{\otimes}\,B$ be the linear map defined by
	\begin{equation}
		\widetilde{\omega}(\vartheta(b_{1}),\; b_{2}\otimes b_{3}) = -\omega(b_{1},\; b_{2}\ast b_{3}), \;\; \forall b_{1}, b_{2}, b_{3}\in B. \label{quad-dual}
	\end{equation}
	Then $(B, \vartheta)$ is a completed Leibniz coalgebra.
\end{lem}
\begin{proof}
	For all $b\in B$ and $b_1 \otimes b_2 \otimes b_3\in B_{i}\otimes B_{j}\otimes B_{k}$,
	we have
	\begin{align*}
		&\;\widetilde{\omega}((\id\hat{\otimes}\vartheta)\vartheta(b) - (\vartheta\hat{\otimes}\id)\vartheta(b) -(\hat{\tau}\hat{\otimes}\,\id) (\id\hat{\otimes}\vartheta)\vartheta(b),\ \
		b_1 \otimes b_2 \otimes b_3)\\
		=&\; \omega(b, b_1 \circ (b_2 \circ b_3) - (b_1 \circ b_2) \circ b_3 - b_2 \circ (b_1 \circ b_3) )\\
		=&\; 0.
	\end{align*}
	Then the nondegeneracy of $\widetilde{\omega}$ shows $(\id\hat{\otimes}\vartheta)\vartheta(b) = (\vartheta\hat{\otimes}\id)\vartheta(b) +  (\hat{\tau}\hat{\otimes}\,\id) (\id\hat{\otimes}\vartheta)\vartheta(b)$.
	Thus, $(B, \vartheta)$ is a completed Leibniz coalgebra.
\end{proof}

\begin{ex}\label{ex:Leib-quad}
	Consider the affine Leibniz algebra $(\widehat{V}_{4}, \circ)$ in Example~\ref{ex:gr-Leib}. 
	There is a bilinear form $\omega$ on $(\widehat{V}_{4}, \circ)$ given by
	\begin{equation*}
		\omega(x\kt^{i}, y\kt^{j})
		=\delta_{i+j, 0}\kappa(x, y), \;\; \forall x, y\in B, i, j\in\bz.
	\end{equation*}
	where $\kappa$ is given by $\kappa(x, y)=1=-\kappa(y, x)$ if $x=v_{1}$, $y=v_{3}$, or $x=v_{2}$, $y=v_{4}$ and all remaining basis pairings zero. 
	It is easy to see that $(\widehat{V}_{4}, \circ, \omega)$ is a quadratic $\bz$-graded Leibniz algebra. 
	Then the completed Leibniz coalgebra $(\widehat{V}_{4}, \vartheta)$ obtained by Lemma~\ref{lem:dual} coincides with the completed Leibniz coalgebra given in Example~\ref{ex:zin-coalg}.
\end{ex}

We now give the notion and results on completed Lie bialgebras.

\begin{thm}\label{thm:bialg}
	Let $(A, \cdot, \Delta)$ be a Zinbiel bialgebra, 
	$(B=\oplus_{i\in\bz}B_{i}, \circ, \omega)$ be a quadratic $\bz$-graded Leibniz algebra and $(\g:=A\otimes B,\; [-,-])$ be the induced Lie algebra from $(A, \cdot)$ and $(B, \circ)$ in Proposition~\ref{pro:L-Z-inflie}.
	Define linear maps $\vartheta: B\rightarrow B\,\hat{\otimes}\,B$ by Eq.~\eqref{quad-dual} and $\delta: \g\rightarrow\g\,\hat{\otimes}\,\g$ by
	Eq.~\eqref{colie}.
	Then $(\g, [-, -], \delta)$ is a completed Lie bialgebra.
	In particular, if $(B=\oplus_{i\in\bz}B_{i}, \circ, \omega)$ is the quadratic $\bz$-graded Leibniz algebra $(\widehat{V}_{4}, \circ, \omega)$ given in Example~\ref{ex:Leib-quad}, then $(\g, [-, -], \delta)$ is a completed Lie bialgebra if and only if $(A, \cdot, \Delta)$ is a Zinbiel bialgebra.
\end{thm}
\begin{proof}
By Lemma~\ref{lem:dual}, $(B, \vartheta)$ is a completed Leibniz coalgebra. 
Then by Proposition~\ref{pro:lie-co}, $(\g, \delta)$ is a completed Lie coalgebra.
Let $a \otimes x, b \otimes y \in A \otimes B$ with $x, y \in B$ homogeneous elements.
Suppose that
\begin{equation*}
	\vartheta(x) = \sum_{i,j,\alpha}x_{1i\alpha} \otimes x_{2j\alpha}, \;\; \vartheta(y) = \sum_{i,j,\alpha}y_{1i\alpha} \otimes y_{2j\alpha}.
\end{equation*}
For all homogeneous elements $u, v \in B$, we have
\begin{align*}
&\;\widetilde{\omega}(\sum_{i,j,\alpha}y_{1i\alpha}\otimes x \circ y_{2j\alpha}-\vartheta(x \circ y)+\sum_{i,j,\alpha}x_{1i\alpha}\otimes x_{2j\alpha} \circ y,\ \ u\otimes v)\\[-2mm]
=&\;\omega(y,\ \ u\circ(x\circ v)-x\circ(u\circ v)-(u\circ x)\circ v) = 0.
\end{align*}
Since $\widetilde{\omega}$ is left nondegenerate, we obtain
\begin{equation*}
	\sum_{i,j,\alpha}y_{1i\alpha}\otimes x \circ y_{2j\alpha}
	=\vartheta (x \circ y)
	-\sum_{i,j,\alpha}x_{1i\alpha}\otimes x_{2j\alpha} \circ y.
\end{equation*}
Similarly, we have
\begin{align*}
	&\sum_{i,j,\alpha} y_{1i\alpha}\circ x \otimes y_{2j\alpha}=0=\sum_{i,j,\alpha}
	 x_{1i\alpha}\circ y \otimes x_{2j\alpha}, \;\;
	\sum_{i,j,\alpha} x\circ y_{1i\alpha}
	\otimes y_{2j\alpha}=\sum_{i,j,\alpha} x_{1i\alpha}\otimes x_{2j\alpha}\circ y , \\
	&\sum_{i,j,\alpha} y\circ x_{1i\alpha} \otimes x_{2j\alpha}=\sum_{i,j,\alpha}
	y_{1i\alpha}\otimes y_{2j\alpha}\circ x , \;\;
	\sum_{i,j,\alpha}x_{1i\alpha}\otimes
	 y\circ x_{2j\alpha} =\vartheta(y\circ x)-\sum_{i,j,\alpha}y_{1i\alpha}
	\otimes y_{2j\alpha}\circ x .
\end{align*}
Thus,
\begin{align*}
	&\; (\ad(a\otimes x)\,\hat{\otimes}\,\id
	+\id\,\hat{\otimes}\,\ad(a\otimes x))\delta(b\otimes y)
	-(\ad(b\otimes y)\,\hat{\otimes}\,\id
	+\id\,\hat{\otimes}\,\ad(b\otimes y))\delta(a\otimes x)\\
	=&\;(\id-\hat{\tau})\bigg(\Big(\sum_{(b)}b_{(1)}\otimes a\cdot b_{(2)}\Big)
	\bullet\vartheta(x\circ y)-\Big(\sum_{(a)}a_{(1)}\otimes b\cdot a_{(2)} \Big)
	\bullet\vartheta(y\circ x)\\[-2mm]
	&\quad+\Big(\sum_{(b)}\big(a\cdot b_{(1)}\otimes b_{(2)}-b_{(1)}\otimes a\cdot b_{(2)} \big)
	+\sum_{(a)}a_{(1)}\otimes a_{(2)}\cdot b \Big)\bullet\Big(\sum_{i,j,\alpha}x_{1i\alpha}
	\otimes x_{2j\alpha}\circ y \Big) \\[-2mm]
	&\quad-\Big(\sum_{(b)}b_{(1)}\otimes b_{(2)}\cdot a +\sum_{(a)}\big( b\cdot a_{(1)} \otimes
	a_{(2)}-a_{(1)}\otimes b\cdot a_{(2)} \big)\Big)\bullet\Big(\sum_{i,j,\alpha}
	y_{1i\alpha}\otimes y_{2j\alpha}\circ x \Big)\bigg)\\
	=&\; (\id-\hat{\tau})\Big(\Delta(a\cdot b)\bullet\vartheta(x\circ y)
	-\Delta(b\cdot a)\bullet\vartheta(y\circ x))\Big)+\Phi\\
	=&\; \delta([a\otimes x,\; b\otimes y])+\Phi,
\end{align*}
where
\begin{align*}
	\Phi:=&\;(\id-\hat{\tau})\bigg(\Big(\sum_{(b)}b_{(1)}\otimes a\cdot b_{(2)}
	-\Delta(a\cdot b)\Big)\bullet\vartheta(x\circ y) - \Big(\sum_{(a)}a_{(1)}\otimes b\cdot a_{(2)}-\Delta(b\cdot a)\Big)
	\bullet\vartheta(y\circ x)\\[-2mm]
	&\;+\Big(\sum_{(b)}\big(a\cdot b_{(1)}\otimes b_{(2)}-b_{(1)}\otimes a\cdot b_{(2)} \big)
	+\sum_{(a)}a_{(1)}\otimes a_{(2)}\cdot b \Big)\bullet\Big(\sum_{i,j,\alpha}x_{1i\alpha}
	\otimes x_{2j\alpha}\circ y \Big) \\[-2mm]
	&\;-\Big(\sum_{(b)}b_{(1)}\otimes b_{(2)}\cdot a +\sum_{(a)}\big( b\cdot a_{(1)} \otimes
	a_{(2)}-a_{(1)}\otimes b\cdot a_{(2)} \big)\Big)\bullet\Big(\sum_{i,j,\alpha}
	y_{1i\alpha}\otimes y_{2j\alpha}\circ x \Big)\bigg).
\end{align*}
Substituting
\begin{equation*}
	\vartheta(y\circ x)=\sum_{i,j,\alpha}y_{1i\alpha}\otimes y_{2j\alpha}\circ x
	-\sum_{i,j,\alpha}x_{1i\alpha}\otimes x_{2j\alpha}\circ y - \hat{\tau}\Big(
	\sum_{i,j,\alpha}y_{1i\alpha}\otimes y_{2j\alpha}\circ x - \sum_{i,j,\alpha}x_{1i\alpha}
	\otimes x_{2j\alpha}\circ y \Big)-\hat{\tau}\vartheta(x\circ y),
\end{equation*}
for $\vartheta(y \circ x)$ in $\Phi$, we obtain
\begin{align*}
\Phi=&\;(\id-\hat{\tau})\bigg(\Big(\sum_{(b)}b_{(1)}\otimes a\cdot b_{(2)} 
-\Delta(a\cdot b)-\sum_{(a)} b\cdot a_{(2)} \otimes a_{(1)}+\tau\Delta(b\cdot a)\Big)
\bullet(\vartheta(x\circ y))\\[-2mm]
&\;+\Big(\sum_{(b)}\big(a\cdot b_{(1)}\otimes b_{(2)}-b_{(1)}\otimes a\cdot b_{(2)} \big)
-\Delta(b\cdot a)-\tau\Delta(b\cdot a)\\[-3mm]
&\qquad\quad+\sum_{(a)}\big(a_{(1)}\otimes a_{(2)}\cdot b + a_{(1)}\otimes b \cdot a_{(2)}
+ b\cdot a_{(2)} \otimes a_{(1)}\big)\Big)\bullet
\Big(\sum_{i,j,\alpha}x_{1i\alpha}\otimes x_{2j\alpha}\circ y \Big)\\[-2mm]
&\;-\Big(\sum_{(b)}b_{(1)}\otimes b_{(2)}\cdot a +\sum_{(a)}\big(
 b\cdot a_{(1)} \otimes a_{(2)}-a_{(1)}\otimes b\cdot a_{(2)} \\[-3mm]
&\qquad\quad+a_{(1)}\otimes b\cdot a_{(2)}+ b\cdot a_{(2)} \otimes a_{(1)}\big)
-\Delta(b\cdot a)-\tau\Delta(b\cdot a)\Big)\bullet
\Big(\sum_{i,j,\alpha}y_{1i\alpha}\otimes y_{2j\alpha}\circ y \Big)\bigg).
\end{align*}
Since $(A, \cdot, \Delta)$ is a Zinbiel bialgebra, we have
\begin{align*}
&\qquad\tau( \Delta(b\cdot a) - \Delta(a\cdot b)+\sum_{(b)}b_{(1)}\otimes a\cdot b_{(2)} 
-\sum_{(a)} b\cdot a_{(2)} \otimes a_{(1)}=0,\\[-2mm]
& \Delta(b\cdot a)+\tau \Delta(b\cdot a) = \sum_{(a)}(a_{(1)}\otimes a_{(2)}\cdot b 
+a_{(1)}\otimes b\cdot a_{(2)} + b\cdot a_{(2)} \otimes a_{(1)})\\[-4mm]
&\qquad\qquad\qquad\qquad\qquad\qquad+\sum_{(b)}( a\cdot b_{(1)} \otimes b_{(2)}
-b_{(1)}\otimes a\cdot b_{(2)} ),\\[-2mm]
& \Delta(b\cdot a)+\tau \Delta(b\cdot a) =\sum_{(a)}( b\cdot a_{(1)} \otimes a_{(2)}
-a_{(1)}\otimes b\cdot a_{(2)} +a_{(1)}\otimes b\cdot a_{(2)} + b\cdot a_{(2)} \otimes a_{(1)}) \\[-4mm]
&\qquad\qquad\qquad\qquad\qquad\qquad
+\sum_{(b)}b_{(1)}\otimes b_{(2)} \cdot a.
\end{align*}
Hence, $\Phi=0$ and therefore $(\g, [-, -], \delta)$ is a completed Lie bialgebra.

Conversely, suppose that $(B=\oplus_{i\in\bz}B_{i}, \circ, \omega)$ is the quadratic $\bz$-graded Leibniz algebra given in Example~\ref{ex:Leib-quad} and $(\g, [-, -], \delta)$ is a completed Lie bialgebra. 
Then $(A, \cdot)$ is a Zinbiel algebra by Proposition~\ref{pro:L-Z-inflie} and $(A, \Delta)$ is a Zinbiel coalgebra by Proposition~\ref{pro:lie-co}. 
Thus, to prove $(A, \cdot, \Delta)$ is a Zinbiel bialgebra, we only need to show that Eqs.~\eqref{zbialg1} and \eqref{zbialg2} hold.
We consider the expression of $\Phi$. 
Setting  $x=v_{1}\kt^{i}$ and $y=v_{2}\kt^{j}$ in $\Phi$ and comparing the coefficients of $\sum_{k}v_{4}\kt^{k}\otimes v_{1}\kt^{i+j-k}$ in the equation $\Phi=0$, we get
\begin{equation*}
	\Delta(a\cdot b)+\Delta(b\cdot a)=(\id\otimes(\fl_{\cdot}+\fr_{\cdot})(b))(\Delta(a))
	+(\fl_{\cdot}(a)\otimes\id)(\Delta(b)).
\end{equation*}
That is, Eq.~\eqref{zbialg1} holds. 
Setting $x=v_{2}\kt^{i}$ and $y=v_{3}\kt^{j}$ in $\Phi$ and comparing the coefficients of $\sum_{k}v_{3}\kt^{k}\otimes v_{4}\kt^{i+j-k}$ in the equation $\Phi=0$, we get
\begin{equation*}
	\Delta(a\cdot b)+\tau(\Delta(a\cdot b))=(\id\otimes\fr_{\cdot}(b))(\Delta(a))
	+(\fl_{\cdot}(a)\otimes\id)(\Delta(b)+\tau(\Delta(b))).
\end{equation*}
That is, Eq.~\eqref{zbialg2} holds. 
Hence, $(A, \cdot, \Delta)$ is a Zinbiel bialgebra. 
The proof is completed.
\end{proof}

%%%%%%%%%%%%%%%%%%%%%%%%%%%%%%%%%%%%%%%%%%%%%%%%%%%%%%%%%%%%%%%%%%%%%%%%%%%
\subsection{Quasi-triangular Lie bialgebras from Zinbiel bialgebras}\label{ssec:qtLi-Zin}
In this subsection, we consider infinite-dimensional Lie bialgebras constructed from quasi-triangular Zinbiel bialgebras, especially triangular Zinbiel bialgebras.
To this end, we study relationships between solutions of the $\ZYBE$ in a Zinbiel algebra and completed solutions of the $\CYBE$ in the induced $\bz$-graded Lie algebra.

\begin{pro}\label{pro:ZYBE-CYBE}
	Let $(A, \cdot)$ be a Zinbiel algebra, $(B=\oplus_{i\in\bz}B_{i}, \circ, \omega)$ be a quadratic $\bz$-graded Leibniz algebra and $(\g=A\otimes B, [-,-])$ be the induced $\bz$-graded Lie algebra. 
	Let $\{e_{\lambda}\}_{\lambda\in\Lambda}$ be a basis of $B$ consisting of homogeneous elements and $\{f_{\lambda}\}_{\lambda\in\Lambda}$ be the dual basis with respect to $\omega$ consisting of homogeneous elements.
	If $r=\sum_{i}x_{i}\otimes y_{i}\in A\otimes A$ is a solution of the $\ZYBE$ in $(A, \cdot)$ such that $r-\tau(r)$ is Zinb-invariant, then the tensor element
	\begin{equation}
		\hat{r}=\sum_{\lambda\in\Lambda}\sum_{i}(x_{i}\otimes e_{\lambda}) \otimes(y_{i}\otimes f_{\lambda})\in\g\,\hat{\otimes}\,\g   \label{cop-sol}
	\end{equation}
	is a completed solution of the $\CYBE$ in $(\g=A\otimes B, [-,-])$ and $\hat{r}
	+\hat{\tau}(\hat{r})$ is Lie-invariant.
	Furthermore, if $(B=\oplus_{i\in\bz}B_{i}, \circ, \omega)$ is the quadratic $\bz$-graded Leibniz algebra given in Example~\ref{ex:Leib-quad},
	then 
	\begin{align}
		\hat{r} = \sum_{k \in \bz}\sum_{i}&\Big((x_{i} \otimes v_{1}\kt^{k}) \otimes (y_{i} \otimes v_{3}\kt^{-k}) + (x_{i} \otimes v_{2}\kt^{k}) \otimes (y_{i} \otimes v_{4}\kt^{-k}) \label{eq:epcop-sol} \\ 
		&\quad - (x_{i} \otimes v_{3}\kt^{k}) \otimes (y_{i} \otimes v_{1}\kt^{-k}) - (x_{i} \otimes v_{4}\kt^{k}) \otimes (y_{i} \otimes v_{2}\kt^{-k})\Big) \notag 
	\end{align}
	is a completed solution of the $\CYBE$ in $(\g, [-,-])$ and $\hat{r} + \hat{\tau}(\hat{r})$ is Lie-invariant if and only if $r$ is a solution of the $\ZYBE$ in $(A, \cdot)$ such that $r-\tau(r)$ is Zinb-invariant.
\end{pro}
\begin{proof}
	For all $s, t, p\in\Lambda$, we have
	\begin{equation*}
		\widetilde{\omega}(\sum_{\lambda,\mu\in\Lambda} e_{\lambda}\circ e_{\mu} \otimes
		f_{\lambda}\otimes f_{\mu},\ \ e_{s}\otimes e_{t}\otimes e_{p})
		=\omega(e_{t}\circ e_{p},\ \ e_{s})=\widetilde{\omega}(\sum_{\lambda,\mu\in\Lambda}
		e_{\lambda}\otimes e_{\mu}\otimes f_{\mu}\circ f_{\lambda} ,\ \
		e_{s}\otimes e_{t}\otimes e_{p}).
	\end{equation*}
	Since $\widetilde{\omega}$ is left nondegenerate, we get
	\begin{equation*}
		\sum_{\lambda,\mu\in\Lambda} e_{\lambda}\circ e_{\mu} \otimes f_{\lambda}\otimes f_{\mu}=\sum_{\lambda,\mu\in\Lambda}e_{\lambda}\otimes
		e_{\mu}\otimes f_{\mu}\circ f_{\lambda}.
	\end{equation*}
	Similarly, we have
	\begin{align*}
		&\sum_{\lambda,\mu\in\Lambda} e_{\mu}\circ e_{\lambda} \otimes
		f_{\lambda}\otimes f_{\mu}=-\sum_{\lambda,\mu\in\Lambda}e_{\lambda}\otimes e_{\mu}\circ
		f_{\lambda} \otimes f_{\mu}, \\
		&\sum_{\lambda,\mu\in\Lambda}e_{\lambda}\otimes f_{\lambda}\circ e_{\mu} \otimes f_{\mu}=-\sum_{\lambda,\mu\in\Lambda}e_{\lambda}
		\otimes e_{\mu}\otimes f_{\lambda}\circ f_{\mu} = \sum_{\lambda,\mu\in\Lambda}
		e_{\lambda} \circ e_{\mu} \otimes f_{\lambda}\otimes f_{\mu}
		+\sum_{\lambda,\mu\in\Lambda} e_{\mu}\circ e_{\lambda} \otimes f_{\lambda}\otimes f_{\mu}.
	\end{align*}
	Thus,
	\begin{align*}
	&\;[\hat{r}_{12}, \hat{r}_{13}]+[\hat{r}_{12}, \hat{r}_{23}]+[\hat{r}_{13}, \hat{r}_{23}]\\
	=&\; \sum_{i,j}\sum_{\lambda,\mu\in\Lambda}\Big(
	( x_{i}\cdot x_{j} \otimes y_{i}\otimes y_{j})\bullet
	(e_{\lambda}\circ e_{\mu}\otimes f_{\lambda}\otimes f_{\mu})
	-( x_{j}\cdot x_{i} \otimes y_{i}\otimes y_{j})\bullet
	(e_{\mu} \circ e_{\lambda}\otimes f_{\lambda}\otimes f_{\mu})\\[-4mm]
	&\qquad\quad\;\; + (x_{i}\otimes y_{i}\cdot x_{j} \otimes y_{j})\bullet
	(e_{\lambda}\otimes f_{\lambda}\circ e_{\mu} \otimes f_{\mu} )
	-(x_{i}\otimes x_{j}\cdot y_{i} \otimes y_{j})\bullet
	(e_{\lambda}\otimes e_{\mu}\circ f_{\lambda} \otimes f_{\mu})\\[-1mm]
	&\qquad\quad\;\; +(x_{i}\otimes x_{j}\otimes y_{i}\cdot y_{j} )\bullet
	(e_{\lambda}\otimes e_{\mu}\otimes f_{\lambda}\circ f_{\mu} )
	-(x_{i}\otimes x_{j}\otimes y_{j}\cdot y_{i} )\bullet
	(e_{\lambda}\otimes e_{\mu}\otimes f_{\mu}\circ f_{\lambda} )\Big) \\
	=&\; \sum_{i,j}\sum_{\lambda,\mu\in\Lambda}\Big((
	x_{i}\cdot x_{j}\otimes y_{i}\otimes y_{j}+x_{i}\otimes y_{i}\cdot x_{j} \otimes y_{j}
	-x_{i}\otimes x_{j}\otimes y_{i}\cdot y_{j} \\[-4mm]
	&\qquad\qquad\qquad\qquad\qquad\qquad-x_{i}\otimes x_{j}\otimes y_{j}\cdot y_{i} 
	)\bullet(e_{\lambda}\circ e_{\mu}\otimes f_{\lambda}\otimes f_{\mu})\\
	&\qquad\qquad-( x_{j}\cdot x_{i}\otimes y_{i}\otimes y_{j}
	-x_{i}\otimes y_{i}\cdot x_{j} \otimes y_{j}
	-x_{i}\otimes x_{j}\cdot y_{i} \otimes y_{j}\\[-1mm]
	&\qquad\qquad\qquad\qquad\qquad\qquad+x_{i}\otimes x_{j}\otimes y_{i}\cdot y_{j} )
	\bullet(e_{\mu}\circ e_{\lambda}\otimes f_{\lambda}\otimes f_{\mu})\Big).
	\end{align*}
	Since $r-\tau(r)$ is Zinb-invariant, we have
	\begin{align*}
	&\sum_{i,j}\big(x_{i}\otimes x_{j}\cdot y_{i} \otimes y_{j}
	+x_{i}\otimes y_{i}\cdot x_{j} \otimes y_{j}
	- x_{j}\cdot x_{i} \otimes y_{i}\otimes y_{j}\\[-5mm]
	&\qquad\quad -y_{i}\otimes x_{j}\cdot x_{i} \otimes y_{j}
	-y_{i}\otimes x_{i}\cdot x_{j} \otimes y_{j}
	+ x_{j}\cdot y_{i} \otimes x_{i}\otimes y_{j}\big)=0,\\
	&\sum_{i,j}\big(x_{i}\otimes x_{j}\otimes y_{j}\cdot y_{i} 
	+x_{i}\otimes x_{j}\otimes y_{i}\cdot y_{j} 
	- y_{j}\cdot x_{i} \otimes x_{j}\otimes y_{i}\\[-5mm]
	&\qquad\quad -y_{i}\otimes x_{j}\otimes y_{j}\cdot x_{i} 
	-y_{i}\otimes x_{j}\otimes x_{i}\cdot y_{j} 
	+ y_{j}\cdot y_{i} \otimes x_{j}\otimes x_{i}\big)=0,
	\end{align*}
	and therefore
	\begin{align*}
	\mathbf{Z}_{r} &= \sum_{i,j}( y_{i}\cdot x_{j} \otimes x_{i}\otimes y_{j}
	-x_{i}\otimes x_{j}\otimes y_{i}\cdot y_{j} -x_{j}\otimes x_{i}\otimes y_{i}\cdot y_{j} 
	+y_{i}\otimes x_{i}\cdot x_{j} \otimes y_{j})\\[-2mm]
	&=\sum_{i,j}(y_{i}\otimes x_{i}\cdot x_{j} \otimes y_{j}
	-y_{i}\otimes x_{j}\otimes y_{j}\cdot x_{i} -y_{i}\otimes x_{j}\otimes x_{i}\cdot y_{j} 
	+ y_{j}\cdot y_{i} \otimes x_{j}\otimes x_{i}).
	\end{align*}
	Thus, 
	\begin{align*}
		&\;[\hat{r}_{12}, \hat{r}_{13}]+[\hat{r}_{12}, \hat{r}_{23}]+[\hat{r}_{13}, \hat{r}_{23}] \\
		=&\;(\tau\otimes\id)(\mathbf{Z}_{r})\bullet(\sum_{\lambda,\mu\in\Lambda} e_{\lambda} \circ e_{\mu} \otimes f_{\lambda}\otimes f_{\mu})-(\id\otimes\tau)(\tau\otimes\id)(\mathbf{Z}_{r})
		\bullet(\sum_{\lambda,\mu\in\Lambda} e_{\mu} \circ e_{\lambda} \otimes f_{\lambda} \otimes f_{\mu}) \\
		=&\; 0.
	\end{align*}
	That is, $\hat{r}$ is a completed solution of the $\CYBE$ in $(\g, [-,-])$.
	Next, we show that $\hat{r}+\hat{\tau}(\hat{r})$ is Lie-invariant.
	Since $\widetilde{\omega}$ is left nondegenerate, by a similar argument, we obtain
	\begin{align*}
		&\sum_{\lambda\in\Lambda} e_{\mu}\circ e_{\lambda} \otimes f_{\lambda}=\sum_{\lambda\in\Lambda}f_{\lambda}\otimes e_{\mu}\circ e_{\lambda}
		=-\sum_{\lambda\in\Lambda}e_{\lambda}\otimes e_{\mu}\circ f_{\lambda} 
		=-\sum_{\lambda\in\Lambda} e_{\mu}\circ f_{\lambda} \otimes e_{\lambda}, \\
		&\sum_{\lambda\in\Lambda} e_{\lambda}\circ e_{\mu} \otimes f_{\lambda}=-
		\sum_{\lambda\in\Lambda} f_{\lambda}\circ e_{\mu} \otimes e_{\lambda}, \\
		&\sum_{\lambda\in\Lambda}e_{\lambda}\otimes f_{\lambda}\circ e_{\mu} =-
		\sum_{\lambda\in\Lambda}f_{\lambda}\otimes e_{\lambda}\circ e_{\mu}=
		\sum_{\lambda\in\Lambda} e_{\mu}\circ e_{\lambda} \otimes f_{\lambda} +
		\sum_{\lambda\in\Lambda} e_{\lambda} \circ e_{\mu} \otimes f_{\lambda},
	\end{align*} 
	for all $\mu\in\Lambda$.
	Therefore, we have
	\begin{align*}
		&\;\big(\id\otimes\ad(a\otimes e_{\mu})+\ad(a\otimes e_{\mu}) \otimes\id\big)(\hat{r}+\hat{\tau}(\hat{r}))\\
		=&\;\sum_{i}\sum_{\lambda\in\Lambda}\Big((x_{i}\otimes a\cdot y_{i})\bullet
		(e_{\lambda}\otimes e_{\mu} \circ f_{\lambda}) - (x_{i}\otimes y_{i}\cdot a )\bullet(e_{\lambda}\otimes f_{\lambda} \circ e_{\mu} )\\[-4mm]
		&\qquad\qquad+( a \cdot x_{i} \otimes y_{i})\bullet( e_{\mu} \circ e_{\lambda}
		\otimes f_{\lambda})
		-( x_{i} \cdot a \otimes y_{i})\bullet( e_{\lambda}\circ e_{\mu} \otimes f_{\lambda})\\
		&\qquad\qquad+(y_{i}\otimes a\cdot x_{i} )\bullet(f_{\lambda}\otimes
		e_{\mu}\circ e_{\lambda} )
		-(y_{i}\otimes x_{i}\cdot a )\bullet(f_{\lambda}\otimes e_{\lambda}\circ e_{\mu} )\\[-1mm]
		&\qquad\qquad+( a\cdot y_{i} \otimes x_{i})\bullet( e_{\mu}\circ f_{\lambda} 
		\otimes e_{\lambda})
		-( y_{i}\cdot a \otimes x_{i})\bullet( f_{\lambda}\circ e_{\mu} \otimes e_{\lambda})\Big) \\
		=&\;\sum_{i}\sum_{\lambda\in\Lambda}\Big(( y_{i}\cdot a \otimes x_{i} +y_{i}\otimes x_{i}\cdot a - x_{i}\otimes y_{i}\cdot a - x_{i}\cdot a \otimes y_{i} )
		\bullet( e_{\lambda}\circ e_{\mu} \otimes f_{\lambda})\\[-4mm]
		&+(a\cdot x_{i}\otimes y_{i}+y_{i}\otimes a\cdot x_{i} 
		+y_{i}\otimes x_{i}\cdot a -x_{i}\otimes a\cdot y_{i} -x_{i}\otimes y_{i}\cdot a - a\cdot y_{i} \otimes x_{i})
		\bullet(e_{\mu}\circ e_{\lambda}\otimes f_{\lambda})\Big).
	\end{align*}
	Since $r-\tau(r)$ is Zinb-invariant, we have
	\begin{align*}
		&\qquad\qquad \sum_{i} y_{i}\cdot a \otimes x_{i}+y_{i}\otimes x_{i}\cdot a 
		-x_{i}\otimes y_{i}\cdot a -  x_{i}\cdot a \otimes y_{i}=0,\\[-2mm]
		&\sum_{i} a\cdot x_{i} \otimes y_{i}+y_{i}\otimes a\cdot x_{i} +y_{i}\otimes x_{i}\cdot a - x_{i}\otimes a\cdot y_{i} -x_{i}\otimes y_{i}\cdot a - a\cdot y_{i} \otimes x_{i}=0.
	\end{align*}
	Therefore, $\big(\id\otimes\ad(a\otimes e_{\mu})+\ad(a\otimes e_{\mu})
	\otimes\id\big)(\hat{r}+\hat{\tau}(\hat{r}))=0$ and $\hat{r}+\hat{\tau}(\hat{r})$ is Lie-invariant.
	
	Suppose that $(B=\oplus_{i\in\bz}B_{i}, \circ, \omega)$ is the quadratic $\bz$-graded Leibniz
	algebra given in Example~\ref{ex:Leib-quad}, $\hat{r}$ is a completed solution of the $\CYBE$ in $(\g, [-,-])$ and $\hat{r}+\hat{\tau}(\hat{r})$ is Lie-invariant.
	Then, by comparing the coefficients of
	$\sum_{i,j,k}v_{4}\kt^{i}\otimes v_{3}\kt^{j}\otimes v_{1}\kt^{-i-j}$ on both sides
	of the equation $[\hat{r}_{12}, \hat{r}_{13}]+[\hat{r}_{12}, \hat{r}_{23}]
	+[\hat{r}_{13}, \hat{r}_{23}]=0$, we can obtain $\mathbf{Z}_{r}=0$.
	By comparing the coefficients of $\sum_{k,l}v_{3}\kt^{k}\otimes v_{1}\kt^{-k}$ on both sides
	of the equation $\big(\id\otimes\ad(a\otimes v_{2})+\ad(a\otimes v_{2})
	\otimes\id\big)(\hat{r}+\hat{\tau}(\hat{r}))=0$, we can obtain
	\begin{equation*}
		\sum_{i}\big(x_{i}\otimes a\cdot y_{i} +x_{i}\otimes y_{i}\cdot a 
		- a\cdot x_{i}\otimes y_{i} - y_{i}\otimes a\cdot x_{i} - y_{i}\otimes x_{i}\cdot a 
		+ a\cdot y_{i} \otimes x_{i}\big)=0.
	\end{equation*}
	That is, $r-\tau(r)$ is Zinb-invariant.
	Hence, we show that $r$ is a solution of the $\ZYBE$ in $(A, \cdot)$ and $r-\tau(r)$
	is Zinb-invariant.
	The proof is finished.
\end{proof}

In particular, for the skew-symmetric completed solution of the
$\CYBE$ in the induced $\bz$-graded Lie algebra, we have:

\begin{cor}\label{cor:sZYBE-sCYBE}
	Let $(A, \cdot)$ be a Zinbiel algebra, $(B=\oplus_{i\in\bz}B_{i}, \circ, \omega)$ be a quadratic $\bz$-graded Leibniz algebra and $(\g=A\otimes B, [-,-])$ be the induced $\bz$-graded Lie algebra. 
	Let $\{e_{\lambda}\}_{\lambda\in\Lambda}$ be a basis of $B$ consisting of homogeneous elements and $\{f_{\lambda}\}_{\lambda\in\Lambda}$ be the dual basis with respect to $\omega$ consisting of homogeneous elements.
	If $r=\sum_{i}x_{i}\otimes y_{i}\in A\otimes A$
	is a symmetric solution of the $\ZYBE$ in $(A, \cdot)$, then $\hat{r}\in\g\,\hat{\otimes}\,\g$ given by Eq.~\eqref{cop-sol} is a skew-symmetric completed solution of the $\CYBE$ in $(\g, [-,-])$. 
	Furthermore, if $(B=\oplus_{i\in\bz}B_{i}, \circ, \omega)$ is the quadratic $\bz$-graded Leibniz algebra given in Example~\ref{ex:Leib-quad}, then $\hat{r}$ given by Eq.~\eqref{eq:epcop-sol} is a skew-symmetric completed solution of the $\CYBE$ in $(\g, [-,-])$ if and only if $r$ is a symmetric solution of the $\ZYBE$ in $(A, \cdot)$.
\end{cor}

\begin{proof}
Note that 
$\sum_{\lambda\in\Lambda}e_{\lambda}\otimes f_{\lambda} =-\sum_{\lambda\in\Lambda}f_{\lambda}\otimes e_{\lambda}$, we have $\hat{r}$ is skew-symmetric. 
Therefore, by Proposition~\ref{pro:ZYBE-CYBE}, we obtain that $\hat{r}$ is a skew-symmetric completed solution of the $\CYBE$ in $(\g, [-,-])$.
Conversely, suppose that $(B=\oplus_{i\in\bz}B_{i}, \circ, \omega)$ is the quadratic $\bz$-graded Leibniz algebra given in Example~\ref{ex:Leib-quad} and $\hat{r}$ is a skew-symmetric completed solution of the $\CYBE$ in $(\g, [-,-])$.
Note that 
\begin{equation*}
	0 = \hat{\tau}(\hat{r}) + \hat{r},
\end{equation*}
we show that $r$ is symmetric by comparing the coefficients of $v_{1}\kt^{k} \otimes v_{3}\kt^{-k}$.
By Proposition~\ref{pro:ZYBE-CYBE} again, we show that $r$ is a symmetric solution of the $\ZYBE$ in $(A, \cdot)$.
\end{proof}

Synthesizing Theorem~\ref{thm:bialg}, Proposition~\ref{pro:ZYBE-CYBE}, and Corollary~\ref{cor:sZYBE-sCYBE} yields another key result of this section.

\begin{thm}\label{thm:indu-sLiebia-zin}
	Let $(A, \cdot, \Delta)$ be a quasi-triangular Zinbiel bialgebra associated with $r$, $(B=\oplus_{i\in\bz}B_{i}$, $\circ$, $\omega)$ be a quadratic $\bz$-graded Leibniz algebra and $(\g=A\otimes B, [-,-])$ be the induced
	$\bz$-graded Lie algebra. 
	Let $\{e_{\lambda}\}_{\lambda\in\Lambda}$ be a basis of $B$ consisting of homogeneous elements and $\{f_{\lambda}\}_{\lambda\in\Lambda}$ be the dual basis with respect to $\omega$ consisting of homogeneous elements.
	Define $\delta$ by Eq.~\eqref{colie} and $\hat{r}$ by Eq.~\eqref{cop-sol}.
	Then $\delta$ coincides $\delta$ defined by Eq.~\eqref{coplie-cobo} through $\hat{r}$.
	Therefore the completed Lie bialgebras $(A\otimes B, [-,-], \delta)$ is quasi-triangular.
	In particular, if $(A, \cdot, \Delta)$ is triangular, then $(A\otimes B, [-,-], \delta)$ is triangular.
\end{thm}
\begin{proof}
	The proof is analogous to that of Theorem~\ref{thm:indu-sLiebia-lei}
\end{proof}

\begin{ex}\label{ex:ind-spelie}
Let $(A, \cdot, \Delta)$ be the triangular Zinbiel bialgebra associated with $r=e_{1}\otimes e_{2}+e_{2}\otimes e_{1}$ given in Example~\ref{ex:ZYBE}.
Consider the quadratic $\bz$-graded Leibniz algebra $(\widehat{V}_{4}, \circ, \omega)$ given in
Example~\ref{ex:Leib-quad}. 
By Theorem \ref{thm:bialg}, we get a completed Lie bialgebra
$(A\otimes B, [-,-], \delta)$, where the non-zero products are given by
\begin{align*}
&[e_{1}\otimes v_{1}\kt^{i},\; e_{1}\otimes v_{2}\kt^{j}]=2e_{2}\otimes v_{1}\kt^{i+j}
=-[e_{1}\otimes v_{2}\kt^{j},\; e_{1}\otimes v_{1}\kt^{i}],\\[-1mm]
&[e_{1}\otimes v_{3}\kt^{i},\; e_{1}\otimes v_{1}\kt^{j}]=e_{2}\otimes v_{4}\kt^{i+j}
=-[e_{1}\otimes v_{1}\kt^{j},\; e_{1}\otimes v_{3}\kt^{i}],\\[-1mm]
&[e_{1}\otimes v_{2}\kt^{i},\; e_{1}\otimes v_{3}\kt^{j}]=e_{2}\otimes v_{3}\kt^{i+j}
=-[e_{1}\otimes v_{3}\kt^{j},\; e_{1}\otimes v_{2}\kt^{i}],
\end{align*}
and the non-zero coproducts are given by
\begin{align*}
&\delta(e_{1}\otimes v_{1}\kt^{i})=\sum_{j}\big((e_{2}\otimes v_{4}\kt^{j})\otimes
(e_{2}\otimes v_{1}\kt^{i-j})-(e_{2}\otimes v_{1}\kt^{j})\otimes
(e_{2}\otimes v_{4}\kt^{i-j})\big),\\[-2mm]
&\delta(e_{1}\otimes v_{2}\kt^{i})=\sum_{j}\big((e_{2}\otimes v_{1}\kt^{j})\otimes
(e_{2}\otimes v_{3}\kt^{i-j})-(e_{2}\otimes v_{3}\kt^{j})\otimes
(e_{2}\otimes v_{1}\kt^{i-j})\big),\\[-2mm]
&\delta(e_{1}\otimes v_{3}\kt^{i})=2\sum_{j}\big((e_{2}\otimes v_{3}\kt^{j})\otimes
(e_{2}\otimes v_{4}\kt^{i-j})-(e_{2}\otimes v_{4}\kt^{j})\otimes
(e_{2}\otimes v_{3}\kt^{i-j})\big).
\end{align*}
which is a triangular Lie bialgebra associated with
\begin{align*}
\hat{r}=&\sum_{i}\big((e_{1}\otimes v_{3}\kt^{i})\otimes(e_{2}\otimes v_{1}\kt^{-i})
+(e_{2}\otimes v_{3}\kt^{i})\otimes(e_{1}\otimes v_{1}\kt^{-i})
+(e_{1}\otimes v_{4}\kt^{i})\otimes(e_{2}\otimes v_{2}\kt^{-i})\\[-4mm]
&\qquad +(e_{2}\otimes v_{4}\kt^{i})\otimes(e_{1}\otimes v_{2}\kt^{-i})
-(e_{1}\otimes v_{1}\kt^{i})\otimes(e_{2}\otimes v_{3}\kt^{-i})
-(e_{2}\otimes v_{1}\kt^{i})\otimes(e_{1}\otimes v_{3}\kt^{-i})\\[-1mm]
&\qquad -(e_{1}\otimes v_{2}\kt^{i})\otimes(e_{2}\otimes v_{4}\kt^{-i})
-(e_{2}\otimes v_{2}\kt^{i})\otimes(e_{1}\otimes v_{4}\kt^{-i})\big)
\end{align*}
\end{ex}

\begin{rmk}\label{rmk:finite}
The factorizable Zinbiel bialgebras were discussed in \cite{Wan}. 
Similar to the proof of Theorem \ref{thm:indu-sLiebia-zin}, we can show that $(A\otimes B, [-,-], \delta)$ is a factorizable Lie bialgebra if $(A, \cdot, \Delta)$ is factorizable Zinbiel bialgebra and $(B, \circ, \omega)$ is a quadratic Leibniz algebra.
Moreover, similar to the discussion
in the last part of Subsection \ref{ssec:qtLi-Lib}, we have the following commutative diagram:
$$
\xymatrix@C=1.4cm@R=0.5cm{
\txt{$(A, \cdot, \Delta)$ \\ {\tiny a triangular Zinbiel bialgebra}}\ar[d] &
\txt{$r$ \\ {\tiny a symmetric solution} \\ {\tiny of the $\ZYBE$ in $(A, \cdot)$}}
\ar[d]\ar[r]\ar[l]   &
\txt{$r^{\sharp}$\\ {\tiny an $\mathcal{O}$-operator of $(A, \cdot)$} \\
{\tiny associated to $(A^{\ast}, \fl_{\cdot}^{\ast}+\fr_{\cdot}^{\ast}, -\fr_{\cdot}^{\ast})$}}
\ar[d]^-{\mbox{$-\otimes r_{0}^{\sharp}$}} \\
\txt{$(A\otimes B, [-,-], \delta)$ \\
{\tiny a triangular Lie bialgebra}}   &
\txt{$\hat{r}$ \\ {\tiny a skew-symmetric solution} \\ {\tiny of the $\CYBE$ in
$(A\otimes B, [-,-])$}} \ar[r]\ar[l]  &
\txt{$\hat{r}^{\sharp}=r^{\sharp}\otimes r_{0}^{\sharp}$ \\
{\tiny an $\mathcal{O}$-operator of $(A\otimes B, [-,-])$ } \\
{\tiny associated to $((A\otimes B)^{\ast}, -\ad^{\ast})$}}}
$$
\end{rmk}

\smallskip
%%%%%%%%%%%%%%%%%%%%%%%%%%%%%%%%%%%%%%%%%%%%%%%%%%%%%%%%%%%%%%%%%%%%%%%%%%%
\subsection{Quasi-Frobenius Lie algebras from quasi-Frobenius Zinbiel algebras}
\label{ssec:QFliealg}
In this subsection, we provide a construction of quasi-Frobenius Lie algebras from quasi-Frobenius Zinbiel algebras and quadratic Leibniz algebras.

\begin{defi}\label{def:qFZib}
Let $(A, \cdot)$ be a Zinbiel algebra. 
If there is a symmetric nondegenerate bilinear form $\varpi$ on $A$ satisfying
\begin{align}
\varpi(a_{1}\cdot a_{2},\; a_{3})-\varpi(a_{1}\cdot a_{3}+a_{3}\cdot a_{1},\; a_{2}) + \varpi(a_{3}\cdot a_{2},\; a_{1})=0, \;\; \forall a_{1}, a_{2}, a_{3}\in A. \label{qFZ}
\end{align}
then $(A, \cdot, \varpi)$ is called a {\bf quasi-Frobenius Zinbiel algebra}.
\end{defi}

There is a close relationship between quasi-Frobenius Zinbiel algebras and solutions of the $\ZYBE$ in Zinbiel algebras.
\begin{pro}\label{pro:QFZib}
	Let $(A, \cdot)$ be a Zinbiel algebra and  $\varpi$ be a nondegenerate bilinear form on $A$.
	Suppose that $\{e_{1}, e_{2},\cdots, e_{n}\}$ is a basis of $A$ and $\{f_{1}, f_{2},\cdots, f_{n}\}$ is the dual basis with respect to $\varpi$.
	Then $r=\sum_{i}e_{i}\otimes f_{i}\in A\otimes A$ is a symmetric solution of the $\ZYBE$ in $(A, \cdot)$ if and only if $(A, \cdot, \varpi)$ is a quasi-Frobenius Zinbiel algebra.
\end{pro}
\begin{proof}
	It is easy to see that $r=\sum_{i}e_{i}\otimes f_{i}$ is symmetric if and only if $\varpi$ is symmetric.
	Suppose that $r$ is symmetric.
	Note that
	\begin{align*}
		r_{13}\cdot r_{21}&=\sum_{i,j} e_{i}\cdot f_{j} \otimes e_{j}\otimes f_{i} 
		=\sum_{i,j} e_{i}\cdot e_{j} \otimes f_{j}\otimes f_{i} 
		=\sum_{i,j,k}\varpi(e_{i}\cdot e_{j},\; f_{k})e_{k}\otimes f_{j}\otimes f_{i},\\[-2mm]
		r_{21}\cdot r_{13}&=\sum_{i,j} f_{j}\cdot e_{i} \otimes e_{j}\otimes f_{i}
		=\sum_{i,j} e_{j}\cdot e_{i} \otimes f_{j}\otimes f_{i} 
		=\sum_{i,j,k}\varpi(e_{j}\cdot e_{i},\; f_{k})e_{k}\otimes f_{j}\otimes f_{i},\\[-2mm]
		r_{12} \cdot r_{23}&=\sum_{i,k}e_{k}\otimes f_{k}\cdot e_{i} \otimes f_{i} =\sum_{i,j,k}\varpi(f_{k}\cdot e_{i},\; e_{j})e_{k}\otimes f_{j}\otimes f_{i},\\[-2mm]
		r_{23}\cdot r_{12}&=\sum_{i,k}e_{k}\otimes e_{i}\cdot f_{k} \otimes f_{i} 
		=\sum_{i,j,k}\varpi(e_{i}\cdot f_{k},\; e_{j})e_{k}\otimes f_{j}\otimes f_{i},\\[-2mm]
		r_{13}\cdot r_{23}&=\sum_{k,j}e_{k}\otimes e_{j}\otimes f_{k}\cdot f_{j} 
		=\sum_{k,j}e_{k}\otimes f_{j}\otimes f_{k}\cdot e_{j}
		=\sum_{i,j,k}\varpi(f_{k}\cdot e_{j},\; e_{i})e_{k}\otimes f_{j}\otimes f_{i},\\[-2mm]
		r_{23}\cdot r_{13}&=\sum_{k,j}e_{k}\otimes e_{j}\otimes f_{j}\cdot f_{k} 
		=\sum_{k,j}e_{k}\otimes f_{j}\otimes e_{j}\cdot f_{k}
		=\sum_{i,j,k}\varpi(e_{j}\cdot f_{k},\; e_{i})e_{k}\otimes f_{j}\otimes f_{i},\\[-2mm]
		r_{13}\cdot r_{12}&=\sum_{i,j} e_{i}\cdot e_{j} \otimes f_{j}\otimes f_{i}
		=\sum_{i,j,k}\varpi(e_{i}\cdot e_{j},\; f_{k})e_{k}\otimes f_{j}\otimes f_{i},\\[-2mm]
		r_{23}\cdot r_{21}&=\sum_{i,k}f_{k}\otimes e_{i} \cdot e_{k} \otimes f_{i}
		=\sum_{i,k}e_{k}\otimes e_{i} \cdot f_{k} \otimes f_{i}
		=\sum_{i,j,k}\varpi(e_{i}\cdot f_{k},\; e_{j})e_{k}\otimes f_{j}\otimes f_{i},
	\end{align*}
	we have
	\begin{align*}
\mathbf{Z}_{r}&=r_{13}\cdot r_{21}+r_{21}\cdot r_{13}+r_{12}\cdot r_{23}+r_{23}\cdot r_{12}
-r_{13}\cdot r_{23}-r_{23}\cdot r_{13}-r_{13}\cdot r_{12}-r_{23}\cdot r_{21}\\
&=\sum_{i,j,k}\Big(\varpi(e_{j}\cdot e_{i},\; f_{k})+\varpi(f_{k}\cdot e_{i},\; e_{j})
-\varpi(f_{k}\cdot e_{j}+e_{j}\cdot f_{k},\; e_{i})\Big)e_{k}\otimes f_{j}\otimes f_{i}.
\end{align*}
Thus, $r$ is a solution of the $\ZYBE$ in $(A, \cdot)$ if and only if Eq.~\eqref{qFZ} holds. 
The proof is completed.
\end{proof}

\cite{HZ} introduced the notion of pre-Zinbiel algebras, established the relationship between pre-Zinbiel algebras and $\mathcal{O}$-operators on a Zinbiel algebra, and also presented some special solutions of the $\ZYBE$.

\begin{defi}[\cite{HZ}]\label{def:pZib}
	A {\bf pre-Zinbiel algebra} is a triple $(A, \tl, \tr)$, where $A$ is a vector space and $\tl, \tr: A \otimes A \rightarrow A$ are binary operations satisfying
	\begin{align*}
	x\tr(y\tr z)&=(x\tr y)\tr z+(x\tl y)\tr z+(y\tr x)\tr z+(y\tl x)\tr z,\\
	x\tr(z\tl y)&=z\tl(x\tr y)+z\tl(x\tl y)=(x\tr z)\tl y+(z\tl x)\tl y, \;\; \forall x, y, z\in A.
	\end{align*}
\end{defi}

Let $(A, \tl, \tr)$ be a pre-Zinbiel algebra, then the binary operation $\cdot: A \otimes A \rightarrow A$ defined by
\begin{equation*}
	x\cdot y=x\tl y+x\tr y, \;\; \forall x, y \in A,
\end{equation*}
endows $A$ with a Zinbiel algebra structure $(A, \cdot)$, called the {\bf sub-adjacent Zinbiel algebra} of $(A, \tl, \tr)$. 
Moreover, both $(A, \fl_{\tr}, \fr_{\tl})$ and $(A^{\ast}, \fl_{\tr}^{\ast}+\fr_{\tl}^{\ast},
-\fr_{\tl}^{\ast})$ are representations of $(A, \cdot)$. 

\begin{pro}[\cite{HZ}]\label{pro:pZib-QF}
Let $(A, \tl, \tr)$ be a pre-Zinbiel algebra and $(A, \cdot)$ be the sub-adjacent Zinbiel algebra.  
Let $\{e_{1}, e_{2},\cdots, e_{n}\}$ be a basis of $A$ and $\{e_{1}^{\ast}, e_{2}^{\ast},\cdots, e_{n}^{\ast}\}\subseteq A^{\ast}$ be the dual basis.
Then
\begin{equation*}
	r:=\sum_{i}(e_{i}\otimes e_{i}^{\ast}+e_{i}^{\ast}\otimes e_{i}),
\end{equation*}
is a symmetric solution of the $\ZYBE$ in $(A \ltimes_{\fl_{\tr}^{\ast}+\fr_{\tl}^{\ast}, -\fr_{\tl}^{\ast}} A^{\ast}, \cdot)$.
\end{pro}

From Propositions~\ref{pro:QFZib} and \ref{pro:pZib-QF}, we can use pre-Zinbiel algebras to provide many examples of quasi-Frobenius Zinbiel algebras, as stated in the following corollary.

\begin{cor}\label{cor:pZib-QF}
Let $(A, \tl, \tr)$ be a pre-Zinbiel algebra and $(A, \cdot)$ be the sub-adjacent Zinbiel algebra. Then $(A \ltimes_{\fl_{\tr}^{\ast}+\fr_{\tl}^{\ast}, -\fr_{\tl}^{\ast}} A^{\ast}, \cdot, \varpi)$ is a quasi-Frobenius Zinbiel algebra with $\varpi$ defined by
\begin{equation}
\varpi((x, \xi),\; (y, \zeta))=\langle\xi, y\rangle+\langle\zeta, x\rangle, \;\; \forall x, y\in A, \xi, \zeta\in A^{\ast}. \label{biform}
\end{equation}
\end{cor}
\begin{proof}
	Let $\{e_{1}, e_{2},\cdots, e_{n}\}$ be a basis of $A$ and $\{e_{1}^{\ast}, e_{2}^{\ast}, \cdots, e_{n}^{\ast}\}$ be the dual basis. 
	Then, $\{e_{1}, \cdots, e_{n}$, $e_{1}^{\ast}, \cdots, e_{n}^{\ast}\}$ is a basis
	of $A \oplus A^{\ast}$, and $\{e_{1}^{\ast}, \cdots, e_{n}^{\ast}, e_{1}, \cdots, e_{n}\}$ is the dual basis of $\{e_{1}, \cdots, e_{n}$, $e_{1}^{\ast}, \cdots, e_{n}^{\ast}\}$ with respect to $\varpi$. 
	By Proposition~\ref{pro:pZib-QF}, $r=\sum_{i}(e_{i}\otimes e_{i}^{\ast}+e_{i}^{\ast}\otimes e_{i})$ is a symmetric solution of the $\ZYBE$ in $(A \ltimes_{\fl_{\tr}^{\ast}+\fr_{\tl}^{\ast}, -\fr_{\tl}^{\ast}} A^{\ast}, \cdot)$.
	Then, by Proposition~\ref{pro:QFZib}, $(A \ltimes_{\fl_{\tr}^{\ast}+\fr_{\tl}^{\ast}, -\fr_{\tl}^{\ast}} A^{\ast}, \cdot, \varpi)$ is a quasi-Frobenius Zinbiel algebra.
\end{proof}

\begin{defi}[\cite{AB, Fis, HBG}]\label{def:Fro-Lie}
Let $(\g=\oplus_{i\in\bz}\g_{i}, [-,-])$ be a $\bz$-graded Lie algebra. 
If there is a skew-symmetric nondegenerate bilinear form $\mathcal{B}$ on $(\g=\oplus_{i\in\bz}\g_{i}, [-,-])$ satisfying
\begin{equation*}
	\mathcal{B}([g_{1}, g_{2}], g_{3})
	+\mathcal{B}([g_{3}, g_{1}], g_{2})+\mathcal{B}([g_{2}, g_{3}], g_{1})=0, \;\; \forall g_{1}, g_{2}, g_{3}\in\g
\end{equation*}
then $(\g=\oplus_{i\in\bz}\g_{i}, [-,-],
\mathcal{B})$ is called a {\bf quasi-Frobenius $\bz$-graded Lie algebra}.
When $\g=\g_{0}$, it is simply called a {\bf quasi-Frobenius Lie algebra} or {\bf symplectic Lie algebra}.
\end{defi}
 
To end this section, we present a construction of quasi-Frobenius Lie algebras through quasi-Frobenius Zinbiel algebras and quadratic $\bz$-graded Leibniz algebras.

\begin{pro}\label{pro:quasi-frob}
Let $(A, \cdot, \varpi)$ be a quasi-Frobenius Zinbiel algebra, $(B, \circ, \omega)$ be a
quadratic $\bz$-graded Leibniz algebra and $(A\otimes B, [-,-])$ be the induced $\bz$-graded
Lie algebra. Define a bilinear form $\mathcal{B}$ on $(A\otimes B, [-,-])$ by
\begin{align}
\mathcal{B}(a_{1}\otimes b_{1},\; a_{2}\otimes b_{2})
=\varpi(a_{1}, a_{2})\omega(b_{1}, b_{2}), \;\; \forall a_{1}, a_{2}\in A, \; b_{1}, b_{2}\in B. \label{QFli}
\end{align}
Then $(A\otimes B, [-,-], \mathcal{B})$ is a quasi-Frobenius $\bz$-graded Lie algebra.
\end{pro}

\begin{proof}
Clearly, $\mathcal{B}$ is skew-symmetric and nondegenerate. 
For all $a_{1}, a_{2}, a_{3}\in A$ and $b_{1}, b_{2}, b_{3}\in B$, we have
\begin{align*}
&\;\mathcal{B}([a_{1}\otimes b_{1},\; a_{2}\otimes b_{2}],\; a_{3}\otimes b_{3})
+\mathcal{B}([a_{3}\otimes b_{3},\; a_{1}\otimes b_{1}],\; a_{2}\otimes b_{2})
+\mathcal{B}([a_{2}\otimes b_{2},\; a_{3}\otimes b_{3}],\; a_{1}\otimes b_{1})\\
=&\;\varpi(a_{1}\cdot a_{2},\; a_{3})\omega(b_{1}\circ b_{2},\; b_{3})
-\varpi(a_{2}\cdot a_{1},\; a_{3})\omega(b_{2}\circ b_{1},\; b_{3})
+\varpi(a_{3}\cdot a_{1},\; a_{2})\omega(b_{3}\circ b_{1},\; b_{2})\\[-1mm]
&\quad-\varpi(a_{1}\cdot a_{3},\; a_{2})\omega(b_{1}\circ b_{3},\; b_{2})
+\varpi(a_{2}\cdot a_{3},\; a_{1})\omega(b_{2}\circ b_{3},\; b_{1})
-\varpi(a_{3}\cdot a_{2},\; a_{1})\omega(b_{3}\circ b_{2},\; b_{1})\\
=&\; \Big(\varpi(a_{1}\cdot a_{3},\; a_{2})+\varpi(a_{2}\cdot a_{3},\; a_{1})
-\varpi(a_{1}\cdot a_{2},\; a_{3})-\varpi(a_{2}\cdot a_{1},\; a_{3})\Big)
\omega(b_{2}\circ b_{3},\; b_{1})\\[-2mm]
&\quad+\Big(\varpi(a_{3}\cdot a_{1},\; a_{2})+\varpi(a_{1}\cdot a_{3},\; a_{2})
-\varpi(a_{1}\cdot a_{2},\; a_{3})-\varpi(a_{3}\cdot a_{2},\; a_{1})\Big)
\omega(b_{3}\circ b_{2},\; b_{1}) \\
=&\; 0.
\end{align*}
Therefore, $(A\otimes B, [-,-], \mathcal{B})$ is a quasi-Frobenius $\bz$-graded Lie algebra.
\end{proof}

Consequently, the following result follows from Propositions~\ref{pro:quasi-frob} and Corollary~\ref{cor:pZib-QF}.
\begin{cor}\label{cor:per-qflie}
	Let $(A, \tl, \tr)$ be a pre-Zinbiel algebra and $(A, \cdot)$ be the sub-adjacent Zinbiel algebra. 
	Let $(A \ltimes_{\fl_{\tr}^{\ast}+\fr_{\tl}^{\ast}, -\fr_{\tl}^{\ast}} A^{\ast}, \cdot, \varpi)$ be the quasi-Frobenius Zinbiel algebra obtained in Corollary~\ref{cor:pZib-QF}. 
	Suppose that $(B, \circ, \omega)$ is a quadratic Leibniz algebra and $((A \ltimes_{\fl_{\tr}^{\ast}+\fr_{\tl}^{\ast}, -\fr_{\tl}^{\ast}} A^{\ast})\otimes B,\; [-,-])$ is the induced Lie algebra. 
	Then $((A \ltimes_{\fl_{\tr}^{\ast}+\fr_{\tl}^{\ast}, -\fr_{\tl}^{\ast}} A^{\ast})\otimes B,\; [-,-],\; \mathcal{B})$ is a quasi-Frobenius Lie algebra, where $\mathcal{B}$ is defined by Eqs.~\eqref{QFli}.
\end{cor}

To finish the paper, we present an example of quasi-Frobenius Lie algebras
obtained by pre-Zinbiel algebras.

\begin{ex}\label{ex:qf-lie}
Let $(A, \tl, \tr)$ be a 2-dimensional pre-Zinbiel algebra with a basis $\{e_{1}, e_{2}\}$ whose non-zero products are defined by
\begin{equation*}
	e_{1}\tr e_{1}=2e_{2}, \;\; e_{1}\tl e_{1}=-e_{2},
\end{equation*}
and $(A, \cdot)$ be the sub-adjacent Zinbiel
algebra.
Let $e^{\ast}_{1}, e^{\ast}_{2}$ be the dual basis of $e_{1}, e_{2}$, and we obtain a quasi-Frobenius Zinbiel algebra $(A \ltimes_{\fl_{\tr}^{\ast}+\fr_{\tl}^{\ast}, -\fr_{\tl}^{\ast}} A^{\ast}, \cdot, \varpi)$, 
where non-zero products of $\cdot$ are defined by
\begin{equation*}
	e_{1} \cdot e_{1}=e_{2}, \; e_{1} \cdot e^{\ast}_{2}=e^{\ast}_{1}=e^{\ast}_{2} \cdot e_{1},
\end{equation*}
and the symmetric nondegenerate bilinear form $\varpi$ is given by $\varpi(e_{i}, e^{\ast}_{j})=\varpi(e^{\ast}_{j}, e_{i}) = \delta_{i, j}$.

Let $(V, \diamond, \omega)$ be a 4-dimensional quadratic Leibniz algebra, where $(V, \diamond)$ is the 4-dimensional Leibniz algebra given in Example~\ref{ex:gr-Leib}, i.e., non-zero product $\diamond$ is given by
\begin{equation*}
	v_{1} \diamond v_{2} = v_{1} = -v_{2} \diamond v_{1}, \;\;  
	v_{1}\diamond v_{3} = -v_{4}, \;\;
	v_{2} \diamond v_{3} = v_{3},
\end{equation*}
and the skew-symmetric invariant nondegenerate bilinear form $\omega$ is defined by
\begin{equation*}
	\omega(v_{1}, v_{3})=1=-\omega(v_{3}, v_{1}),\; 
	\omega(v_{2}, v_{4})=1=-\omega(v_{4}, v_{2}). 
\end{equation*}
By Corollary~\ref{cor:per-qflie}, we obtain a 16-dimensional quasi-Frobenius Lie algebra $((A \ltimes_{\fl_{\tr}^{\ast}+\fr_{\tl}^{\ast}, -\fr_{\tl}^{\ast}} A^{\ast})\otimes B,\; [-,-],\; \mathcal{B})$, 
where the non-zero bracket is given by
\begin{align*}
&[e_{1}\otimes v_{1},\; e_{1}\otimes v_{2}]=2e_{2}\otimes v_{1}
=-[e_{1}\otimes v_{2},\; e_{1}\otimes v_{1}],\quad \\
&[e_{1}\otimes v_{2},\; e_{1}\otimes v_{3}]=e_{2}\otimes v_{3}
=-[e_{1}\otimes v_{3},\; e_{1}\otimes v_{2}],\\[-1mm]
&[e_{1}\otimes v_{3},\; e_{1}\otimes v_{1}]=e_{2}\otimes v_{4}
=-[e_{1}\otimes v_{1},\; e_{1}\otimes v_{3}],\quad\;\; \\
&[e_{1}\otimes v_{3},\; e^{\ast}_{2}\otimes v_{1}]=e^{\ast}_{1}\otimes v_{4}
=-[e^{\ast}_{2}\otimes v_{1},\; e_{1}\otimes v_{3}],\\[-1mm]
&[e_{1}\otimes v_{1},\; e^{\ast}_{2}\otimes v_{2}]=2e^{\ast}_{1}\otimes v_{1}
=-[e^{\ast}_{2}\otimes v_{2},\; e_{1}\otimes v_{1}],\quad \\
&[e_{1}\otimes v_{2},\; e^{\ast}_{2}\otimes v_{3}]=e^{\ast}_{1}\otimes v_{3}
=-[e^{\ast}_{2}\otimes v_{3},\; e_{1}\otimes v_{2}],
\end{align*}
and the skew-symmetric nondegenerate bilinear form $\mathcal{B}$ is given by
\begin{align*}
&\mathcal{B}(e_{1}\otimes v_{1},\; e^{\ast}_{1}\otimes v_{3})=1
=-\mathcal{B}(e^{\ast}_{1}\otimes v_{3},\; e_{1}\otimes v_{1}),\quad
\mathcal{B}(e_{1}\otimes v_{2},\; e^{\ast}_{1}\otimes v_{4})=1
=-\mathcal{B}(e^{\ast}_{1}\otimes v_{4},\; e_{1}\otimes v_{2}),\\[-1mm]
&\mathcal{B}(e_{2}\otimes v_{1},\; e^{\ast}_{2}\otimes v_{3})=1
=-\mathcal{B}(e^{\ast}_{2}\otimes v_{3},\; e_{2}\otimes v_{1}),\quad
\mathcal{B}(e_{2}\otimes v_{2},\; e^{\ast}_{2}\otimes v_{4})=1
=-\mathcal{B}(e^{\ast}_{2}\otimes v_{4},\; e_{2}\otimes v_{2}),
\end{align*}
and others are all zero.
\end{ex}

\bigskip
\noindent
{\bf Acknowledgements. } This work was financially supported by National
Natural Science Foundation of China (No. 11771122).

 \end{document}